\documentclass[12pt]{amsart}
\usepackage{mathrsfs}
\usepackage{amsmath, amssymb, amsthm, latexsym}
\usepackage{amssymb,amscd,amsmath}
\usepackage{latexsym}
\usepackage{MnSymbol}
\usepackage[center]{caption}
\usepackage{tikz}
\newcounter{braid}
\newcounter{strands}

\DeclareMathAlphabet{\bsf}{OT1}{cmss}{bx}{n}

\pgfkeyssetvalue{/tikz/braid height}{1cm}
\pgfkeyssetvalue{/tikz/braid width}{1cm}
\pgfkeyssetvalue{/tikz/braid start}{(0,0)}
\pgfkeyssetvalue{/tikz/braid colour}{black}
\pgfkeys{/tikz/strands/.code={\setcounter{strands}{#1}}}

\makeatletter
\def\cross{%
  \@ifnextchar^{\message{Got sup}\cross@sup}{\cross@sub}}

\def\cross@sup^#1_#2{\render@cross{#2}{#1}}

\def\cross@sub_#1{\@ifnextchar^{\cross@@sub{#1}}{\render@cross{#1}{1}}}

\def\cross@@sub#1^#2{\render@cross{#1}{#2}}

\def\render@cross#1#2{
  \def\strand{#1}
  \def\crossing{#2}
  \pgfmathsetmacro{\cross@y}{-\value{braid}*\braid@h}
  \pgfmathtruncatemacro{\nextstrand}{#1+1}
  \foreach \thread in {1,...,\value{strands}}
  {
    \pgfmathsetmacro{\strand@x}{\thread * \braid@w}
    \ifnum\thread=\strand
    \pgfmathsetmacro{\over@x}{\strand * \braid@w + .5*(1 - \crossing) * \braid@w}
    \pgfmathsetmacro{\under@x}{\strand * \braid@w + .5*(1 + \crossing) * \braid@w}
    \draw[braid] \pgfkeysvalueof{/tikz/braid start} +(\under@x pt,\cross@y pt) to[out=-90,in=90] +(\over@x pt,\cross@y pt -\braid@h);
    \draw[braid] \pgfkeysvalueof{/tikz/braid start} +(\over@x pt,\cross@y pt) to[out=-90,in=90] +(\under@x pt,\cross@y pt -\braid@h);
    \else
    \ifnum\thread=\nextstrand
    \else
     \draw[braid] \pgfkeysvalueof{/tikz/braid start} ++(\strand@x pt,\cross@y pt) -- ++(0,-\braid@h);
    \fi
   \fi
  }
  \stepcounter{braid}
}

\tikzset{braid/.style={double=\pgfkeysvalueof{/tikz/braid colour},double distance=1pt,line width=2pt,white}}

\newcommand{\braid}[2][]{%
  \begingroup
  \pgfkeys{/tikz/strands=2}
  \tikzset{#1}
  \pgfkeysgetvalue{/tikz/braid width}{\braid@w}
  \pgfkeysgetvalue{/tikz/braid height}{\braid@h}
  \setcounter{braid}{0}
  \let\sigma=\cross
  #2
  \endgroup
}
\makeatother

\input xypic
\newtheorem{theorem}{Theorem}
\newtheorem{proposition}[theorem]{Proposition}

\newtheorem{lemma}[theorem]{Lemma}

\newtheorem{corollary}[theorem]{Corollary}


\def\dash{\text{-}}

\def\Z{\mathbb{Z}}

\def\N{\mathbb{N}}

\def\Pi{\mathbb{P}^{\infty}}

\def\qed{\hfill$\square$\medskip}

\def\Zpk{\mathbb{Z}/p^{k}}
\def\Zpk1{\mathbb{Z}/p^{k-1}}

\newcommand{\rref}[1]{(\ref{#1})}

\newcommand{\beg}[2]{\begin{equation}\label{#1}#2\end{equation}}
\def\r{\rightarrow}

\def\sl2{\widetilde{SL_{2}(\Z)}}

\title[Kan's combinatorial spectra]{Kan's combinatorial spectra and their sheaves revisited}
\author{Ruian Chen, Igor Kriz and Ale\v{s} Pultr}
\thanks{The second author acknowledges support from
grant GA\,\v{C}R P201/12/G028 in the Czech Republic, the Simons Foundation, and Charles University in Prague.
The second and third author both acknowledge support
from  the project GA\,\v{C}R P202/12/G061 in the Czech Republic.}

\begin{document}

\maketitle

\begin{abstract}
We define a right Cartan-Eilenberg structure on the category of Kan's combinatorial spectra, and the category of
sheaves of such spectra, assuming some conditions. In both structures, we use the geometric concept of homotopy
equivalence as the strong equivalence. In the case of sheaves, we use local equivalence as the weak equivalence. 
This paper is the first step in a larger-scale program of investigating sheaves of spectra from a geometric viewpoint.
\end{abstract}

\section{Introduction}

\subsection{The stable homotopy category and sheaves}
The stable homotopy category (also known as the stable category, or derived category of spectra)
is a foundational setting for generalized homology and cohomology, 
and as such, is perhaps the most important concept of modern algebraic topology.
Yet, the category does not have a canonical construction, unlike, say, the category of
chain complexes, which plays an analogous role for ordinary
(co)homology. In contrast, different approaches to the stable category exist, each of
which has some advantages and some disadvantages. An extensive foundational
and calculational treatment of the stable category was given by Adams \cite{adams}. But Adams'
treatment does not give an underlying ``strict category", which is often
needed in constructions, just as actual chain complexes, and not
just the objects of their derived category, are needed in homological algebra. 
A strict category of spectra, very closely analogous to the category of topological spaces, 
is provided by {\em May spectra} \cite{lms}, which, notably, also works equivariantly
for compact Lie groups. The May category has a number of more recent
improvements, many of which are related to constructing a point set level commutative associative
smash product \cite{ekmm, hovey, mandell}. 

\vspace{3mm}
A completely different construction of the stable homotopy category can be given using a
concept of a {\em combinatorial spectrum} discovered much earlier by Kan \cite{kan},
which can be described as a ``naive stabilization" of a simplicial set. 
While very appealing aesthetically, this approach has not had nearly as much
follow-up as constructions based on topological
spaces. A part of the reason is that 
even defining a smash product of a Kan spectrum with a based simplicial set (which is necessary
in treating generalized homology of spaces) is difficult, due to the fact that the smash product
of based simplicial sets does not commute with suspension of combinatorial spectra.

\vspace{3mm}
Yet, combinatorial spectra have some advantages. 
Notably, K.S. Brown \cite{ksb} developed a fully functional theory of sheaves of combinatorial spectra which 
is a generalization of abelian sheaves, and can be used to define generalized sheaf cohomology. 
Brown's category of sheaves of combinatorial spectra was also used by Piacenza
\cite{piacenza} to treat {\em locally constant sheaves of spectra}, which is
an approach to {\em parametrized spectra}. A rigorous definition of the derived category of
parametrized spectra was a notoriously hard problem. Treatments based on May spectra were more recently given in
\cite{hu} and \cite{ms}, and those also work for compact Lie groups. The construction
\cite{piacenza} can be used to construct the derived category of parametrized spectra
as a full subcategory of the derived category of sheaves of combinatorial spectra.
On the other hand, a fully functional
category of sheaves of May spectra, beyond locally constant, has so far not been constructed. 
Perhaps the difficulty with the smash product of combinatorial spectra is heuristically related to the ease with which
they are sheafified: for example, left derived functors are less natural
in abelian sheaves also, since abelian sheaves do not have enough projectives.

\vspace{3mm}
In \cite{ksb1}, Brown and Gersten applied the results of \cite{ksb} to algebraic K-theory, which
was later used by Thomason \cite{thomason}. Most of the discussion of sheaf theory concepts for spectra since
that time used Thomason's approach (see e.g. the survey paper \cite{geisser} for examples).
Thomason noticed that given certain hypotheses on the site, 
one can mostly get by with presheaves, by using cosimplicial Godement resolutions, which can
be constructed on the level of presheaves (since they only use stalks). For Godement resolutions,
one only needs a category with directed colimits and products. Applying these techniques, Thomason 
\cite{thomason} used
the category of presheaves of ``$\Omega$-spectra" of Bousfield and Friedlander \cite{bf} (which do not have
arbitrary limits) to define his
version of generalized sheaf cohomology. The category \cite{bf} is not canonical, many variants 
give the same result.  
In general, however, any presheaf approach to sheaf theory is somewhat ``Ersatz": 
for example, it does not have full functoriality with respect to functors which cannot be computed on presheaves,
such as the direct image. Such examples, using the original Brown theory \cite{ksb,ksb1}, 
occurred in the work of Gillet on the Riemann-Roch theorem in K-theory \cite{gillet, gillet1}.

\vspace{3mm}
An important point is that the sheaves of combinatorial spectra defined by K.S.Brown \cite{ksb} remain the only 
category of sheaves of spectra with a rigorously defined stable homotopy category
to date, which is why we consider this setting in the 
present paper. It is {\em not possible} to even discuss sheaves of the
Bousfield-Friedlander 
$\Omega$-spectra \cite{bf} used by Thomason \cite{thomason} because the category of Bousfield-Friedlander
$\Omega$-spectra does not have limits: A Bousfield-Friedlander 
$\Omega$-spectrum is a sequence of based simplicial sets
$Z_n$ together with weak equivalences $Z_n\r \Omega Z_{n+1}$. It is easily seen that a limit (even an intersection)
does not preserve the equivalences. 

\vspace{3mm}
A claim of a construction
of a homotopy theory of sheaves of May spectra has been made by Block and Lazarev \cite{blaz}, and in fact, 
\cite{blaz} claim that their construction works also in other categories based on May spectra, for example
the symmetric monoidal category of $S$-modules \cite{ekmm}. These categories, of course, have limits, 
so one can discuss sheaves,
but the construction of {\em sheafification} (i.e. left adjoint to the forgetful functor from sheaves to
presheaves) is difficult, in fact using the non-elementary framework of Freyd and Kelly \cite{fkelly}. It is
not clear how this sheafification interacts with homotopy. It was claimed in \cite{blaz} without proof 
that it behaves well on 
a class of CW-sheaves, and to our knowledge, no proof has since appeared in the literature.

\vspace{3mm}
\subsection{Cartan-Eilenberg structures}
In the present paper, we revisit Kan's combinatorial spectra,
and their sheaves, in view of the new axiomatic approach to homotopy theory
given in \cite{guillen}, (see also \cite{rg}), the main concept of which is that of a {\em Cartan-Eilenberg structure}. 
The established approach to the foundations of homotopy theory uses the concept of a
{\em Quillen model structure} \cite{quillen}. A Quillen model structure gives a computable construction of the derived
category, and access to left and right derived functors for certain functors known as derivable functors 
(\cite{hirsch}, Section 8.4). 
Describing a Quillen model structure on a category has become the standard method of constructing a
derived category in homotopy theory. Yet, constructing Quillen structures can be often technical and non-canonical:
Different model structures may exist, which may, for example,
describe the same derived category, but may differ in derivable functors, 
so different model structures may actually be needed to making different functors derived.

\vspace{3mm}
The paper \cite{guillen} 
formalized a framework which preceded Quillen 
model structures.  
This framework was first used in 1949 by J.H.C.Whitehead \cite{w1,w2} to develop
the derived category of spaces, and was also
used to construct the stable homotopy category from May spectra \cite{lms}, although 
in both cases, model structures exist also. 
In 1956, Cartan and Eilenberg \cite{ce} used Whitehead's approach, and its dual, to construct
$Ext$ groups in abelian categories with enough projectives, resp. enough injectives. 
In \cite{guillen}, this construction 
was formalized in a general context. For a category $\mathscr{C}$ and a class of morphisms in $
\mathscr{D}\subseteq\mathscr{C}$, by the {\em derived category} (if one exists) we shall mean
the universal category (on the same class of objects) in which the morphisms in $\mathscr{D}$ become
isomorphisms.

\vspace{3mm}
\noindent
{\bf Definition} \cite{guillen}:
Suppose $\mathscr{C}$ is a category together with two classes of morphisms $\mathscr{S}\subseteq\mathscr{E}$ 
called {\em strong and weak homotopy equivalences}. For simplicity, we shall assume that 
both $\mathscr{S}$ and $\mathscr{E}$ contain all isomorphisms, and satisfy the 2/3 axiom (i.e. in a 
commuting triangle of 
morphisms, if two out of three morphisms are in the class, so is the third). 
Suppose further that the derived category of $\mathscr{C}$ with respect to strong homotopies exists. (We
call it the {\em strong homotopy category}.)

Suppose now there is a class $\mathscr{B}\subseteq Obj(\mathscr{C})$
such that for any $y\in \mathscr{B}$, any weak equivalence $f:x_1\r x_2$ induces a bijection on sets of morphisms 
into $y$ (resp. from $y$) in the strong homotopy category. (In that case, we say that $\mathscr{B}$ is {\em local}
(resp, {\em co-local}) with respect to weak equivalences in the strong homotopy category.)
Suppose further that for every $x\in Obj(\mathscr{C})$, there exists 
a weak equivalence $x \r x^\prime$ (resp. $x^\prime\r x$) with $x^\prime\in 
\mathscr{B}$. Then we call the category $\mathscr{C}$ together 
with the data just specified a {\em right (resp. left) Cartan-Eilenberg category}.

\vspace{3mm}
The authors of \cite{guillen} prove that for any right or left Cartan-Eilenberg category, a derived category with respect to
weak homotopy equivalences (called the {\em weak homotopy category})
exists, and is equivalent to the full subcategory of the strong
homotopy category on $\mathscr{B}$. 
In addition, a structure of a left or right Cartan-Eilenberg category is often 
technically easier to prove than a Quillen model structure.

\vspace{3mm}
Often, the Cartan-Eilenberg formalism can be applied in a situation where there is
an underlying ``point set category", and the morphisms of $\mathscr{C}$ are equivalence classes
with respect to a congruence relation of {\em ``naive" homotopy}. Also, in a left resp. right Cartan-Eilenberg structure,
derived functors can be computed from those functors which send strong homotopy equivalences to
isomorphisms. Thus, Cartan-Eilenberg structures can be used 
as a foundation of homotopy theory, without using Quillen model structures.

\vspace{3mm}
One thing to notice is that while a Quillen model structure is {\em absolute}, a Cartan-Eilenberg structure is
{\em relative} in the sense that it needs the strong homotopy category to be defined first before defining the weak
one by means of localization or co-localization. 
Left or right
Cartan-Eilenberg structures can be ``composed"
in the sense that the weak homotopy equivalences of one structure can be the strong homotopy equivalences of
another one, thus defining a ``composite" Cartan-Eilenberg structure. Similar localization and co-localization 
procedures are also possible within Quillen model structure, given some additional conditions \cite{hirsch}.
Also, there is of course always a ``trivial" left (and right) Cartan-Eilenberg structure, where strong and weak homotopy equivalences are the same.

\vspace{3mm}
We see from this discussion that Cartan-Eilenberg structures on a category with given weak homotopy equivalences 
can have quite a different significance based on how the strong homotopy equivalences are chosen. Generally
speaking, the smaller the class of strong homotopy equivalences is, the more powerful the machinery is (since
it makes, for example, more functors derivable).
It is remarked in \cite{rg}, Example 2.11, any Quillen model structure gives rise to both a left and right 
Cartan-Eilenberg structure. The left Cartan-Eilenberg structure, for example, 
is the full subcategory of fibrant objects where
the weak homotopy equivalences are equivalences. The strong homotopy equivalences can be chosen as 
left homotopy equivalences, but stronger (=smaller) choices may be possible. For example, in the case of spaces with the
weak Quillen model structure, all objects are fibrant, and left homotopy equivalences are just weak equivalences,  
so this choice gives nothing new, but we can also choose homotopy equivalences in the ``naive" sense as
strong homotopy equivalences, which is how the Quillen model structure is constructed.
Dually, the Quillen category of simplicial sets, where every object is cofibrant, 
is right Cartan-Eilenberg (\cite{guillen}) where the strong homotopy equivalences are, again, the ``naive"
homotopy equivalences which are used in defining the Quillen structure. 
Generally, our intuition is that right Cartan-Eilenberg categories are more suitable for 
sheafification, just as one can develop a good derived category of sheaves on an abelian category with enough
injectives. 

\vspace{3mm}
In the case of spectra, which are a middle ground between spaces and abelian categories,
May spectra (and also the S-modules of \cite{ekmm}) are left Cartan-Eilenberg with
respect to homotopy equivalences (while
the symmetric spectra of \cite{hovey} are neither left nor right Cartan-Eilenberg). The left Cartan-Eilenberg
property turns out to be very valuable in imitating algebraic structures on spectra \cite{hks}, but it
is not immediately suitable for a fully functorial sheaf theory. A right Cartan-Eilenberg structure is needed. 
In \cite{rg}, such a structure was found on presheaves of Thomason spectra where the 
strong homotopy equivalences there are section-wise weak equivalences of presheaves. 

\vspace{3mm}
This begs the following question:
{\em Is there a right 
Cartan-Eilenberg theory of spectra (Eckmann-Hilton dual to May spectra), and a right Cartan-Eilenberg theory
of their sheaves, where strong equivalences are naive homotopy equivalences (defined by actual homotopies)?}

\vspace{3mm}
In this paper, we answer these questions in the affirmative. However, it appears hopeless to 
use any variation of the construction \cite{bf} for this purpose. This is, roughly speaking, because in a stable Quillen model structure, 
an arbitrary prespectrum cannot be cofibrant -- one needs its structure maps to be cofibrations.
On the other hand, Kan's combinatorial spectra do work.
We prove that
Kan's combinatorial spectra are right Cartan-Eilenberg with respect to homotopy equivalence (Corollary \ref{crg}). 
We also
prove that over a sufficiently nice site, the category of sheaves of combinatorial spectra is
right Cartan-Eilenberg with respect to homotopy equivalences (Theorem \ref{stcegodement}). In fact, this does not seem to be in
the literature even for sheaves of simplicial sets, so we also prove that first (Theorem \ref{tcegodement}).

\vspace{3mm}
\subsection{Comparison with previous work} 
This paper draws heavily on the papers \cite{ksb,guillen,rg}, the first of which
first considered sheaves of Kan spectra and their cohomology, and the second two of which axiomatized
Cartan-Eilenberg structures, and gave many examples. As remarked above, the significance of a Cartan-Eilenberg 
structure depends on its class of strong homotopy equivalences. In this paper, we construct Cartan-Eilenberg structures
whose strong homotopy equivalences are ``naive" homotopy equivalences coming from a notion of homotopy involving
an interval-like object. Preserving such geometric homotopy is
a condition which can be readily verified on many functors, thus giving rise to a large class of derivable functors. 
Additionally, it is the most natural choice, as a basic desideratum of homotopy theory is for naive homotopy-preserving
functors to be derivable.
As far as we know, our right Cartan-Eilenberg structure on Kan's combinatorial 
spectra where strong homotopy equivalences are the naive homotopy equivalences
is new, and no such structure was previously known. 

We then use our right Cartan-Eilenberg structure on Kan's combinatorial spectra to obtain a 
right Cartan-Eilenberg structure on their sheaves where, again, strong homotopy equivalences are homotopy equivalences
with respect to a naive notion of homotopy, coming from an interval-like construction. Again, preserving such 
homotopy is a readily verifiable condition for many functors of interest. 
Our construction is new and no such
construction was previously known. Additionally, Kan spectra 
are the only currently known rigorous setting of homotopy theory of sheaves (rather than presheaves) of spectra,
which makes our result an important part of the foundations.
In \cite{rg}, a right Cartan-Eilenberg structure is given on categories of
sheaves where strong homotopy equivalences are section-wise equivalences. This result of \cite{rg} is used in the proof 
of our main result, which, however, is substantially stronger. While naive homotopy equivalences
are always
section-wise equivalences, the converse is certainly false: for example, over a point, our Cartan-Eilenberg
structure contains the result on Kan
combinatorial spectra, while local equivalences (used in the Cartan-Eilenberg structure of \cite{rg})
and equivalences on sections coincide.

\vspace{3mm}
The present paper is organized as follows: In Section \ref{scomb}, we review Kan's construction of
combinatorial spectra \cite{kan}, and develop some additional technical concepts. We also prove that they
are right Cartan-Eilenberg with respect to homotopy equivalence. In Section \ref{scr}, we
make some observations on cosimplicial realization which we need later. In Section \ref{ssh},
we discuss ``nice" sites and prove the right Cartan-Eilenberg property with respect to homotopy equivalence
for sheaves of simplicial sets and combinatorial spectra.

\section{Combinatorial spectra}\label{scomb}

Combinatorial spectra were introduced by D. Kan in \cite{kan}. Recall the simplicial category $\Delta$ whose
object set is $\N_0=\{0,1,2,\dots\}$ and $\Delta(m,n)$ is the set of maps 
\beg{escat}{\rho:\mathbf{m}=\{0,\dots,m\}\r \mathbf{n}=\{0,\dots,n\}}
preserving $\leq$. There is a self-functor
$$\Phi:\Delta\r\Delta$$
where $\Phi(\mathbf{m})=\mathbf{m+1}$ and for $\rho$ as in \rref{escat}, $\Phi(\rho)$ coincides with $\rho$ on
$\mathbf{m}$, and 
$$(\Phi(\rho))(m+1)=n+1.$$
The category $\Delta_{st}$ is the (strict) colimit in the category of small categories
of the diagram
\beg{edphi}{\diagram
\Delta\rto^\Phi & \Delta\rto^\Phi &\dots
\enddiagram}
Therefore, one can identify the category $\Delta_{st}$ with the category whose object set is $\Z$,
and morphisms are generated by ``faces'' $d_i:\mathbf{m}\r \mathbf{m+1}$ and ``degeneracies'' $s_i:\mathbf{m}
\r \mathbf{m-1}$ which
satisfy the usual simplicial relations. 

Denote by $\Phi^{\infty-n}$ the inclusion of the $n$'th term $\Delta$ of \rref{edphi} into the colimit
$\Delta_{st}$. Note that one can have $n\in\Z$. Also note that we can identify $\Delta_{st}(k,\ell)$ with 
the set of $\leq$-preserving maps $f:\N_0\r\N_0$ which are of the form
\beg{ephinalpha}{f=\Phi^{\infty-n}(\alpha)}
for some $\alpha\in \Delta(k+n,\ell+n)$ where \rref{ephinalpha} is defined as the extension of $\alpha$
given by 
$$\Phi^{\infty-n}(\alpha)(s)=\ell-k+s \;\text{for $s>k+n$}$$
(i.e., put in another way, which satisfy $f(s+1)=f(s)+1$ for $s$ large enough). From this point of view,
{\em faces} (resp. {\em degeneracies}) in the wider sense are morphisms in $\Delta_{st}$ which are injective
(resp. onto) as maps $\N_0\r\N_0$. These are, of course, precisely those morphisms which are compositions
of the maps $d_i$ (resp. $s_i$), $i\geq 0$.

\vspace{3mm}
One denotes by $Set_\bullet$ the category of based sets, whose objects are based sets (sets with a distinguished
base point $*$), and morphisms are mappings preserving $*$. Consider the category $\Delta^{Op}\dash Set_\bullet$
of based simplicial sets, which is the category of functors $\Delta^{Op}\r Set_\bullet$ and natural 
transformations. Then there is a functor
$$\Omega^k:\Delta^{Op}\dash Set_\bullet\r \Delta^{Op}\dash Set_\bullet$$
where for a simplicial set $T:\Delta^{Op}\r Set_\bullet$, 
$$(\Omega^k(T))(n)=\{x\in T(n+k)\mid d_{n+1}(x)=\dots= d_{n+k}(x)=*\}$$
and for $x\in (\Omega^k(T))(n)$, the operators $s_i$, $d_i$, $j\leq  n$ 
act on $x$ the same way as in $T$. We write $\Omega$ instead of $\Omega^1$. We have
$$\Omega^{k+\ell}=\Omega^k\Omega^\ell.$$
The functor $\Omega^k$ has a left adjoint denoted by $\Sigma^k$. 

We can describe the functor $\Sigma$
explicitly as follows, separating the roles of faces 
and degeneracies (a similar description also holds for $\Sigma^k$): Let $\Delta_0$ be the subcategory of 
$\Delta$ consisting of the same objects and the morphisms \rref{escat}
such that 
\beg{errho}{|\rho^{-1}(n)|\leq 1.}
(Note that the category $\Delta_0$ contains the image of $\Phi$, and morphisms in the image of $\Phi$ are 
precisely those which satisfy the inequality in \rref{errho}.)
Then we have a functor
$$\Sigma_0:\Delta^{Op}\dash Set_\bullet \r 
\Delta^{Op}_{0}\dash Set_\bullet$$
where $\Sigma_0Z(n)=Z(n-1)$ for $n\geq 1$ and $\Sigma_0(Z)(0)=*$, and on
$\Sigma_0Z$, morphisms of the form $\Phi(\rho)$ act the same way as $\rho$ on $Z$, and
other morphisms act by $*$.
If we denote the inclusion functor $\iota:\Delta^{Op}_{0}\r\Delta^{Op}$ and its left Kan extension
by $\iota_\sharp$, then
$$\Sigma = \iota_\sharp \Sigma_0.$$

D.Kan \cite{kan} proves the following

\begin{proposition}\label{p1}
The functor $\Sigma^k$ commutes with the simplicial realization functor $|?|$ up to canonical natural
isomorphism, where $\Sigma^k$ on
based spaces denotes the canonical suspension $?\wedge S^k$.
\end{proposition}

\begin{proof}
It suffices to consider $k=1$.
Consider the case of the standard $n$-simplex $\Delta_{n+}$ where $\Delta_n$ is the representable
simplicial set, $\Delta_n(k)=\Delta(k,n)$. Then $\Sigma(\Delta_{n+})$ has a non-degenerate 
element $x\in \Sigma(\Delta_{n+})(n+1)$ which satisfies $d_{n+1}(x)=d_0^n(x)=*$. Clearly, the
geometric realization of this is canonically identified with $\Sigma|\Delta_{n+}|$, and the identification
is compatible with faces and degeneracies, and thus applies to every simplicial set.
\end{proof}

\begin{corollary}\label{c1}
We have a natural inclusion 
\beg{ec1}{|\Omega(X)|\subseteq \Omega|X|}
on based simplicial sets $X$.
\end{corollary}

\begin{proof}
We have a map given by simplicial realization of the counit of adjunction:
$$\Sigma|\Omega(X)|=|\Sigma\Omega(X)|\r |X|.$$
The map \rref{ec1} is its adjoint. One easily verifies that it is an inclusion.
\end{proof}

\begin{proposition}\label{p0}
For a based simplicial set $Z$, the unit of adjunction
\beg{eetaiso}{\eta:Z\r \Omega\Sigma Z}
is an isomorphism.
\end{proposition}

\begin{proof}
Every element of a simplicial set is uniquely expressible as an iterated degeneracy (meaning 
a composition of degeneracies)
of a non-degenerate element (i.e. one which is not in the image of a degeneracy). Now we have a bijection $b$
from the non-degenerate elements of $Z(n)$ to the non-degenerate elements
of $\Sigma Z(n+1)$ for $n\geq 0$ (since this is by definition
true for $\Sigma_0$, and $\iota_\sharp$ is an inclusion on each $?(n)$, and all of the elements in its image
are degenerate. 

On the other hand, we also have a bijection between the non-degenerate elements of $\Omega T(n)$
to the non-degenerate elements of $T(n+1)$ which satisfy $d_{n+1}(x)=*$
for any based simplicial set $T$. This is because if, for $x\in T(n+1)$, $d_{n+1}(x)=*$, then
$x=\rho (y)$ for $y$ non-degenerate where $\rho $ is an iterated degeneracy satisfying \rref{errho} (with $n$ replaced
by $n+1$). Thus, if $x\in \Omega T(n)$ is non-degenerate if and only if $x\in T(n+1)$ is.

Now if we denote for any simplicial set $Z$ by $Z_{nd}$ the sequence of sets of non-degenerate elements,
we see that for $x\in Z_{nd}(n)$, $b(x)\in (\Sigma Z)_{nd}(n+1)$ satisfies $d_{n+1}(b)=*$. Thus, $\eta$
preserves non-degenerate elements, and we have a commutative diagram
$$\diagram
Z_{nd}(n)\rrto^\cong\drto_{\eta_{nd}} && (\Sigma Z)_{nd}(n+1)\\
& (\Omega\Sigma Z)_{nd}(n).\urto_\subseteq &
\enddiagram$$
Thus, $\eta_{nd}$ is a bijection, which implies the statement of the Proposition.

\end{proof}

Similarly, one can consider the category $\Delta^{Op}_{st}\dash Set_\bullet$ of functors
$\Delta^{Op}_{st}\r Set_\bullet$ and natural transformations, called the category of stable simplicial
sets. One has functors
$$\Omega^{\infty+k}:\Delta^{Op}_{st}\dash Set_\bullet\r \Delta^{Op}\dash Set_\bullet$$
where for 
$$(\Omega^{\infty+k}(T))(n)=\{x\in T(n+k)\mid d_{m}(x)=*\text{ for $m>n$}\}.$$
Clearly, we have 
$$\Omega^\ell\Omega^{\infty+k}=\Omega^{\infty+k+\ell},$$
and the functor $\Omega^{\infty+k}$ has a left adjoint $\Sigma^{\infty+k}$.

The category $\mathscr{S}$ of {\em combinatorial spectra} is the full subcategory of 
$\Delta^{Op}_{st}\dash Set_\bullet$ on all stable simplicial sets $T$ with the property that
for all $x\in T(n)$ there exists a $k$ such that $d_m(x)=*$ for $m>k$. This is a coreflexive
subcategory of $\Delta^{Op}_{st}\dash Set_\bullet$ (by passing to the subset of all elements
satisfying the condition). Therefore, the category $\mathscr{S}$ has all limits and colimits.

\begin{proposition}\label{p2}
The category $\mathscr{S}$ is canonically equivalent to the category whose objects are sequences
of simplicial sets $(Z_n)_{n\in \N_0}$ ($\N_0$ can also be equivalently replaced with $\Z$)
together with isomorphisms of simplicial sets
\beg{ecsloop}{\diagram\rho_n:Z_n\rto^\cong &\Omega Z_{n+1}
\enddiagram}
and morphisms are
sequences of morphisms of simplicial sets commuting with the structure maps.
\end{proposition}

\begin{proof}
For a combinatorial spectrum $Z$, put
$$Z_n=\Omega^{\infty-n}Z.$$
For a sequence $Z_n$ with structure maps \rref{ecsloop},
define 
$$Z(n)=\operatornamewithlimits{colim}_k Z_{k}(n+k).
$$
By definition, these functors are inverse to each other up to canonical natural isomorphisms.
\end{proof}

In view of Corollary \ref{c1}, Proposition \ref{p2} gives the sequence $(|Z_n|)$
a functorial structure of an inclusion prespectrum. Denote by $\mathscr{L}(Z)$ the associated
May spectrum. A morphism $f:X\r Y$ of combinatorial spectra is called a {\em (weak) equivalence}
if $\mathscr{L}(f)$ is an equivalence of May spectra.

\vspace{3mm}
Also, Proposition \ref{p2} suggests to also consider the concept of a
{\em combinatorial prespectrum} $Z$ which is a sequence
of morphisms of based simplicial sets
\beg{ecsloop1}{\rho_n:Z_n\r \Omega Z_{n+1},\; n\in \Z
}
(or, by adjunction equivalently, $\Sigma Z_n\r Z_{n+1}$), without any additional assumptions on $\rho_n$. A morphism 
of combinatorial prespectra
$f:Z\r T$
is a sequence of morphisms $f_n:Z_n\r T_n$ such that we have commutative diagrams
$$\diagram
Z_n\dto_{\rho_n}\rto^{f_n} &
T_n\dto^{\rho_n}\\
\Omega Z_{n+1}\rto^{\Omega f_{n+1}}& \Omega T_{n+1}.
\enddiagram$$
Denote the category of combinatorial prespectra by $\mathscr{P}$. We then have a ``forgetful'' functor
$$Ps:\mathscr{S}\r\mathscr{P},$$
which has a left adjoint
$$Sp:\mathscr{P}\r\mathscr{S}.$$
In the notation of Proposition \ref{p2}, for a prespectrum $Z$, we have
$$(Sp(Z))_n=\operatornamewithlimits{colim}\Omega^k Z_{n+k}.$$
Similarly as in the context of May spectra, a variant of the construction of $\mathscr{P}$ is obtained by replacing
the indexing set $\Z$ with $\N_0$ in \rref{ecsloop1}. The resulting category will be denoted
by $\mathscr{P}_0$ and the analogues of the functors $Ps$, $Sp$ by $Ps_0$, $Sp_0$. There is a canonical forgetful
functor $\mathscr{P}\r \mathscr{P}_0$. It is not an equivalence of categories.

\vspace{3mm}

For a combinatorial spectrum $X$, an element $x\in X(n)$, $x\neq *$, is called {\em non-degenerate} if $x$ is not in the image of a degeneracy.

\vspace{3mm}
\begin{lemma}\label{nondeg}
Let $Z$ be a combinatorial spectrum and let $x\in Z(n)$, $x\neq *$. Then there exists a unique degeneracy in the
wider sense $s\in \Delta_{st}(n,m)$ and a unique element $y\in Z(m)$ such that $x=s(y)$ and $y$ is
non-degenerate.
\end{lemma}

\begin{proof}
Suppose $d_s(x)=*$ for $s>N$. By the proof of Proposition \ref{p2}, we may consider $x$ as an
element of $Z_{n-N}(N)$ which is an ordinary (based) simplicial set, where an analogous 
statement is well known,
existence and uniqueness. $N$ is of course not uniquely determined, but if also $x=s^\prime(y^\prime)$
for a degeneracy $s^\prime$ and a non-degenerate element $y^\prime$, we can use the larger of the
$N$'s for $y$ and $y^\prime$ to see that $y=y^\prime$, $s=s^\prime$.
\end{proof}

The sphere spectrum $S$ is the
free combinatorial spectrum on one element $\alpha\in S(0)$ with the relation that $d_i(\alpha)=*$ for
$i\geq 0$. We have
$$S(n)=\{\alpha_n,*\} \;\text{for $n\geq 0$}$$
where $\alpha_n$ is the iterated degeneracy of $\alpha$,  and
$$S(n)=\{*\}\;\text{for $n<0$}.$$
It is immediate from the definition that
$$S=\Sigma^\infty S^0.$$
In the notation of Proposition \ref{p2}, the based simplicial sets $S_n$ are the free based simplicial sets
on one element $\alpha_n\in S_n(n)$ such that $d_i(\alpha_n)=*$ for all $i\geq 0$. In particular, we have
$$\Omega S_n\cong S_{n-1}.$$

\vspace{3mm}
Recall that the standard $n$-simplex is the based simplicial set
$$\Delta_n=\Delta(?,n)$$
(the representable functor). 
As usual, we denote by $\Delta_n^{\circ}$ the simplicial set obtained from $\Delta_n$ by deleting the
non-degenerate element $\alpha\in\Delta_n(n)$ and all its degeneracies, and by $V_{n,k}$ the simplicial set
obtained by deleting, additionally, $d_k(\alpha)$ and all its degeneracies. 

A {\em relative combinatorial cell spectrum} is a morphism $f:X\r Y$ of combinatorial spectra
such that there exist combinatorial spectra $Y_{(m)}$, $m\geq -1$ 
where $Y_{(-1)}=X$,  indexing sets $I_m$, indexing morphisms $n_m:I_m\r\N_0$, $\ell_m:I_m\r \Z$
and pushout diagrams of combinatorial spectra
\beg{ecell1}{\diagram
\displaystyle\bigvee_{i\in I_m} \Sigma^{\infty+\ell_m(i)}\Delta_{n_m(i)+}^\circ\rto^(.6){f_m} \dto_\subseteq & Y_{(m-1)}\dto\\
\displaystyle\bigvee_{i\in I_m} \Sigma^{\infty+\ell_m(i)}\Delta_{n_m(i)+}\rto & Y_{(m)}
\enddiagram
}
such that
$$Y=\lim_\r Y_{(m)}.$$
Here $\bigvee$ denotes the coproduct.

An anodyne extension is defined in the same way as a relative combinatorial cell spectrum 
where in \rref{ecell1}, $\Delta_{n_m(i)}^\circ$ is replaced by $V_{n_m(i), k_m(i)}$. Explicitly, an
{\em anodyne extension} is a morphism $f:X\r Y$ of combinatorial spectra
such that there exist combinatorial spectra $Y_{(m)}$, $m\geq -1$ 
where $Y_{(-1)}=X$,  indexing sets $I_m$, indexing morphisms $n_m:I_m\r\N_0$, $\ell_m:I_m\r \Z$
and pushout diagrams of combinatorial spectra
\beg{ecell1}{\diagram
\displaystyle\bigvee_{i\in I_m} \Sigma^{\infty+\ell_m(i)}V_{n_m(i), k_m(i)+}
\rto^(.6){f_m} \dto_\subseteq & Y_{(m-1)}\dto\\
\displaystyle\bigvee_{i\in I_m} \Sigma^{\infty+\ell_m(i)}\Delta_{n_m(i)+}\rto & Y_{(m)}
\enddiagram
}
such that
$$Y=\lim_\r Y_{(m)}.$$

Clearly,
anodyne extensions are weak equivalences.

\vspace{3mm}
\begin{proposition}\label{pcell}
Every injective morphism of combinatorial spectra is a relative combinatorial cell spectrum. In particular,
every combinatorial spectrum is a cell combinatorial spectrum.
\end{proposition}

\begin{proof}
The only non-degenerate element of $\Sigma^{\infty+\ell}\Delta_{n+}$ which is not in
$\Sigma^{\infty+\ell}\Delta_{n+}^\circ$ is in dimension $n+\ell$, corresponding to the
non-degenerate element of $\Delta_{n}$ which is not in $\Delta_{n}^\circ$. For a non-degenerate
element $x\in Y(n)\smallsetminus X(n)$, let the {\em degree} of $x$ be the minimum $k$ such that 
$d_i(x)=*$ for $i>k$. Then we can let $Y_{(k)}(n)$ be the set of all elements $x$ of the form 
$s(y)$ where $s$ is a degeneracy in the wider sense, and $y$ is non-degenerate of degree $\leq k$. 
By Lemma \ref{nondeg}, $Y_{(k)}$ is a subspectrum of $Y$, and moreover, $Y_{(k+1)}(n)$ is 
obtained from $Y_{(k)}$ by attaching a cell $\Sigma^{\infty-n}(\Delta_{k+})$
for every non-degenerate element of $Y(n)$ of degree $k$.
\end{proof}

\vspace{3mm}
It is important to note that if we can attach a cell of the form $\Sigma^{\infty-\ell}\Delta_{n+}$
to a combinatorial spectrum, we can obtain an {\em isomorphic} combinatorial spectrum by attaching a cell of the form
$\Sigma^{\infty-\ell-k}\Delta_{(n+k)+}$ instead for any $k\geq 0$. Also, a combinatorial cell spectrum
(and hence every combinatorial spectrum) is naturally a directed colimit of inclusions of its finite cell subspectra, which
are, by adjunction, shift desuspensions of simplicial sets. Therefore, in particular, every combinatorial
spectrum is a directed direct limit of inclusions of shift desuspensions.

\vspace{3mm}
For based simplicial sets $K,T$, we have a simplicial set
$$F(K,T)=\Delta^{Op}\dash Set_\bullet(K\wedge (\Delta_n)_+, T)$$
where $(?)_+$ means attaching a disjoint base point. The left adjoint to this functor is
$K\wedge ?$.

\begin{lemma} \cite{kan}\label{lsusp}
There is a canonical natural morphism of based simplicial sets
\beg{esusp0}{(\Sigma K)\wedge T\r \Sigma(K\wedge T)}
which is a weak equivalence (i.e. becomes a homotopy equivalence after applying simplicial realization).
\end{lemma}

\begin{proof}
Consider again the case when both $K=\Delta_{m+}$, $T=\Delta_{n+}$. The non-degenerate elements
in dimension $m+n$ of the simplicial set
\beg{esusp1}{\Delta_{m+}\wedge \Delta_{n+}=(\Delta_{m}\times \Delta_{n})_+}
correspond to shuffles of $m$ ordered elements and $n$ other ordered elements, i.e. their number is
$${m+n\choose m}.$$
After suspension, there will be the same number of nondegenerate elements in dimension $m+n+1$. 
In 
\beg{esusp2}{(\Sigma\Delta_{m+})\wedge \Delta_{n+},}
on the other hand, the non-degenerate elements in dimension $m+n+1$ correspond to
shuffles of $m+1$ ordered elements and $n$ ordered elements (the same as if we replaced
$\Sigma\Delta_{m+}$ by $(\Delta_{m+1})_+$, i.e. their number is
$${m+n+1\choose n}.$$
The morphism from \rref{esusp2} to the suspension of \rref{esusp1} is obtained by applying the last degeneracy
($s_{m+1}$) to all non-degenerate elements of dimension $m+n+1$ corresponding to the shuffles of the
$m+1$ and $n$ ordered elements where the last of the $m+1$ elements is not in the end. One verifies that 
this recipe is natural in the simplicial category. Additionally, by definition, after simplicial realization, the
map \rref{esusp0} becomes a quasi-fibration \cite{dt} with contractible fibers, so it is a weak equivalence between
CW-complexes, and hence a homotopy equivalence. 
\end{proof}

\vspace{3mm}
It is well known that if 
a functor $F_1:C\r D$ has a right adjoint $G_1$ and a functor $F_2:C\r D$ has a right adjoint $G_2$, and we have
a natural transformation $F_1\r F_2$, then we have a natural map of sets
$$D(F_2X,Y)\r D(F_1X,Y),$$
which is, by adjunction,
$$C(X,G_2Y)\r C(X, G_1Y)$$
which, by the Yoneda lemma, gives a canonical natural transformation
$$G_2\r G_1.$$
Applying this principle to the situation of Lemma \ref{lsusp}, the right adjoint to $F_1=(\Sigma ?)\wedge T$
is $G_1=\Omega F(T,?)$, and the right adjoint to $F_2=\Sigma(?\wedge T)$ is $G_2=F(T,\Omega (?))$, so we get a 
canonical natural transformation
\beg{enattr}{F(T,\Omega Y)\r \Omega F(T,Y)}
which, moreover, is injective, since it is right adjoint to a surjective map. 

\vspace{3mm}
This means that if $Z$ is a combinatorial prespectrum and $T$ is a based simplicial set, then we get canonical
morphisms
\beg{eftz}{F(T,Z_n)\r \Omega F(T,Z_{n+1}),}
and since $\Omega$ commutes with colimits of sequences,
a combinatorial prespectrum, denoted by $F_p(T,Z)$. We may of course go on to define
a spectrum $F(T,Z)$ by 
$$F(T,Z)=Sp(F_p(T,Z)),$$
i.e. by setting
\beg{eftzz}{F(T,Z)_n=\operatornamewithlimits{colim}_k \Omega^{k}F(T,Z_{n+k}).}
This seems particularly natural for a combinatorial spectrum $Z$, where one sees that the morphisms
\rref{eftz} are inclusions. 

It is important to note, however, that the functor $F_p(T,?):\mathscr{P}\r\mathscr{P}$ has a left adjoint 
$T\cdot ?$, while the functor
$F(T,?)$ does not. For a prespectrum $Z$ and a based simplicial set $T$, the based
simplicial set $(T\cdot Z)_n$ is the colimit of the diagram of based simplicial sets
\beg{edotcolim}{\diagram
&&\vdots\dto\\
&\Sigma(T\wedge \Sigma Z_{n-2})\dto\rto &\Sigma^2(T\wedge Z_{n-2})\\
T\wedge \Sigma Z_{n-1}\dto\rto &\Sigma(T\wedge Z_{n-1}) &\\
T\wedge Z_n &&
\enddiagram
}
An analogue of these constructions also exists when replacing $\mathscr{P}$ with $\mathscr{P}_0$. It is
interesting to note that in that case, the diagram \rref{edotcolim} is finite for each $n$, containing only the
terms involving $Z_{n-k}$ for $k\leq n$.

In Diagram \rref{edotcolim}, the horizontal morphisms are weak equivalences of simplicial sets, but unless
we know something about the vertical arrows, unfortunately this does not appear to imply anything about the
colimit of \rref{edotcolim}. On the other hand, if $Z$ is a 
combinatorial spectrum, the vertical arrows of \rref{edotcolim} are injective, so the inclusion of 
$T\wedge Z_n$ into the colimit is a weak equivalence.
Inspecting non-degenerate elements in \rref{edotcolim}, we obtain the following 

\begin{lemma}\label{ldotps}
If $Z$ is a combinatorial spectrum, then $?\star Z=Sp(?)\cdot Ps(Z)$ (where ``$?\cdot?$" was defined
under \rref{eftzz} above) preserves weak equivalences as well as injective 
morphisms
in the ``$?$" coordinate.
\end{lemma}

\qed

\vspace{3mm}
\noindent
{\bf Remark:}
For a combinatorial spectrum $Z$ and a based simplicial set $K$, it is possible to describe $K\star Z$ explicitly
in terms of the cell decomposition of $Z$ given in the proof of Proposition \ref{pcell}. We define
$K\star Y_{(k)}$ by induction on $k$ so that for every non-degenerate cell 
\beg{estarell}{e:\Sigma^{\infty-n}\Delta_{\ell+}\r Y_{(k)}}
of degree $\ell\leq k$, we have a morphism
\beg{estarellk}{K\star e:\Sigma^{\infty-n}K\wedge \Delta_{\ell+}\r K\star Y_{(k)}}
such that for an iterated face $d:\Delta_m\r\Delta_\ell$, if we denote by $e^\prime$ 
the non-degenerate cell of $Y_{(k)}$ of degree $p\leq m$,
we have a commutative diagram
\beg{estarcons}{\diagram
\Sigma^{\infty-n}\Sigma^{m-p}K\wedge \Delta_{p+}\rto^{K\star e^\prime} & Y_{(k)}\\
\Sigma^{\infty-n}K\wedge \Sigma^{m-p}\Delta_{p+}\uto &\\
\Sigma^{\infty-n} K\wedge \Delta_{m+}\uto\rto_{\Sigma^{\infty-n}K\wedge d} &\Sigma^{\infty-n}
K\wedge\Delta_{\ell+}\uuto_{K\star e}
\enddiagram
}
where the left column is projection followed by an iteration of the morphism \rref{esusp0} of Lemma \ref{lsusp}.
We may start with $k=-1$, $Y_{(-1)}=*$. We have 
$K\star Y_{(-1)}=*$. Assuming $K\star Y_{(k)}$ has been defined, and $e$ is a cell \rref{estarell} of 
$Y_{(k+1)}$ of degree $\ell=k+1$, we have an attaching map
\beg{estaratt}{\Sigma^{\infty-n}K\wedge \Delta_{\ell+}^\circ\r K\star Y_{(k)}}
given by using the map \rref{estarellk} with $e$ replaced by the faces of $e$, which is
consistent because of Diagram \rref{estarcons}.
Thus, we may push out \rref{estaratt} with the inclusion
$$\Sigma^{\infty-n}K\wedge \Delta_{\ell+}^\circ\r\Sigma^{\infty-n}K\wedge \Delta_{\ell+}.$$
Note that this description can also be taken as a definition of $K\star Z$. Functoriality follows by applying
iterations of the morphisms \rref{esusp0}, since morphisms of combinatorial spectra do not increase the
degree of cells. Immediately from this construction, there follows 

\begin{lemma}\label{tdotps1}
We have a natural (in all coordinates) isomorphism of spectra
\beg{ecanstar}{(T_1\wedge T_2)\star Z\cong T_1\star(T_2\star Z)
}
\end{lemma}

\qed

\vspace{3mm}

Note that $S^0=*_+$ (i.e. the simplicial set of two points with one of them as base point)
is a left unit for the operation $?\cdot?$ (see the paragraph under \rref{eftzz} above), hence also for $\star$.
Now denoting $I=\Delta_1$, a {\em homotopy} $h:f\simeq g$ two morphisms of combinatorial spectra
\beg{efgh}{f,g:X\r Y}
is a morphism
\beg{efghi}{h:I_+\star X\r Y}
such that
$hd_0=f$, $hd_1=g$. The equivalence relation of homotopy $\simeq$ on $\mathscr{S}(X,Y)$ is defined
as the smallest
equivalence relation containing the relation of existence of a homotopy $h:f\simeq g$. By functoriality of $\star$, this is
a congruence relation on the category of combinatorial spectra, and the corresponding quotient category
is denoted by $h\mathscr{S}$, and called the {\em (strong) homotopy category} of combinatorial spectra. An 
isomorphism in $h\mathscr{S}$ is called a {\em homotopy equivalence}. 

\vspace{3mm}
\begin{lemma}\label{shomotopy}
The inclusions
$$d_i:X=S^0\star X\r I_+\star X$$
are injective morphisms, and weak equivalences. A homotopy equivalence of combinatorial
spectra is a weak equivalence.
\end{lemma}

\begin{proof}
The second statement clearly follows from the first. The first statement follows from the fact that $d_i:S^0\r I_+$
are injective morphisms and weak equivalences.
\end{proof}

A morphism $f:X\r Y$ of combinatorial spectra is called a {\em Kan fibration} if in the following commutative
diagram, for any choice of horizontal arrows, the diagonal arrow exists:
\beg{etestfib}{\diagram
\Sigma^{\infty+\ell}(V_{n,k})_+\rto\dto_\subset &
X\dto^f\\
\Sigma^{\infty+\ell}\Delta_{n+}\rto\urdotted|>\tip &Y.
\enddiagram
}
One has a similar lifting
whenever the left vertical arrow is replaced by any anodyne extension. A combinatorial
spectrum $X$ is called {\em Kan fibrant} if the terminal morphism $X\r *$ is a Kan fibration. This is, by adjunction,
equivalent to $X_n$ being a Kan fibrant simplicial set for every $n\in \Z$.

The right adjoint to $\mathscr{L}$ is the functor $Sing$
which takes a May spectrum $T=(T_n)$ to a combinatorial spectrum $Z$ where $Z_n=Sing(T_n)$.
(Note that $\Omega$ commutes with $Sing$, since their left adjoints $\Sigma$ and $|?|$ commute.)
We shall write 
\beg{ezsing}{Z=Sing(T).}

\begin{lemma}\label{lsingsp}
(1) A combinatorial spectrum of the form \rref{ezsing} for a May spectrum $T$ is Kan fibrant. 

(2) For any combinatorial spectrum $Z$, the unit of adjunction
\beg{ezsing1}{\eta: Z\r Sing(\mathscr{L}(Z))
}
is a weak equivalence. 
\end{lemma}

\begin{proof}
For (1), it suffices to show that the based simplicial set $\Omega^{\infty-n}Sing(T)=Sing(T_n)$ are
Kan fibrant, which is well known. 

For (2), since weak equivalences are defined by applying $\mathscr{L}$, it suffices to show that
the canonical morphism
$$\mathscr{L} Sing\mathscr{L}(Z)\r \mathscr{L}(Z)$$
is a weak equivalence. More generally, we claim that the counit of adjunction
$$\epsilon:\mathscr{L}Sing (T)\r T$$
is a weak equivalence for any May spectrum $T$. But $\mathscr{L}Sing(T)$ is the spectrification
of the inclusion prespectra whose terms are $|Sing(T_n)|$, so the statement follows from the fact
that the counit of adjunction
$$|Sing(T_n)|\r T_n$$
is a weak equivalence.
\end{proof}

\vspace{3mm}
\begin{proposition}\label{panod}
For any injective morphism of combinatorial spectra $f:X\r Y$ which is a weak equivalence, there
exists a (necessarily injective) morphism of combinatorial spectra $g:Y\r Z$ such that $g\circ f$ is
an anodyne extension.
\end{proposition}

\begin{proof}
By Proposition \ref{pcell}, $Y$ may be expressed as
$Y_\alpha$ for some ordinal $\alpha$ where we let $Y_0=X$, for a limit ordinal $\beta$, we put
$$Y_\beta=\bigcup_{\gamma<\beta} Y_\gamma,$$
and for every $\beta<\alpha$, $X_{\beta+1}$ is obtained from $X_\beta$ by attaching a cell of the form
$\Sigma^{\infty-\ell}\Delta_{n+}$ by $\Sigma^{\infty-\ell}\Delta_{n+}^\circ$. We will construct, by induction,
inclusions of spectra
\beg{eanodind}{Y_\beta\subseteq X_\beta
}
such that the inclusion $X\subseteq X_\beta$ is an anodyne extension. We just take unions at limit ordinals,
and $X_0=X$, so it suffices to show how to construct $X_{\beta+1}$ from $X_\beta$. 

To this end, first, note that we can assume that $X_\beta$ is Kan fibrant by the ``small object argument" 
(attaching all possible $\Sigma^{\infty-\ell}\Delta_{n+}$'s via different $\Sigma^{\infty-\ell}V_{n,k+}$'s in 
$\omega$ steps). Now consider the attaching map 
$$f:\Sigma^{\infty-\ell}\Delta_{n+}^\circ\r X,$$
and its adjoint
$$\phi:\Delta_{n+}^\circ\r X_\ell.$$
Since $(X_\beta)_\ell$ is a Kan fibrant simplicial set, and 
$$X_\ell\subseteq Y_\ell$$
is a weak equivalence,
$\phi$ extends to a morphism of based simplicial sets
$$\Delta_{n+}\r (X_\beta)_\ell,$$
and hence $f$ extends to a morphism of Kan spectra
$$\Sigma^{\infty-\ell}\Delta_{n+}\r X_\beta.$$
In other words, ``the cell we are trying to attach was already in $X_\beta$''. Thus, instead of the cell, we can
attach 
$$\Sigma^{\infty-\ell}(\Delta_n\times I)_+$$
via 
$$\Sigma^{\infty-\ell}(\Delta_n^\circ\times I\cup \Delta_n\times \{0\})_+,$$
which is an anodyne extension.
\end{proof}

\begin{theorem}\label{tcw}
Kan fibrant combinatorial spectra are local with respect to weak equivalences in $h\mathscr{S}$. In other
words, if we write $[X,Y]=h\mathscr{S}(X,Y)$, then for a weak equivalence $e:X\r Y$ and a
Kan fibrant spectrum $Z$,
\beg{elocale}{[e,Z]:[Y,Z]\r[X,Z]
}
is a bijection. 
\end{theorem}

\vspace{3mm}
\noindent
{\bf Remark:} It is already known, and it also follows from this theorem and Lemma \ref{lsingsp} that
the weak homotopy category of combinatorial spectra is the stable category, and that $[X,Z]$ calculates
morphisms in the weak homotopy category for $Z$ Kan fibrant. Therefore, in particular,
\rref{elocale} is actually an
isomorphism of abelian groups.

\begin{proof}
To prove that \rref{elocale} is onto, first note that without loss of generality, we may assume that $e$ is injective. 
To this end, consider the mapping cylinder
\beg{emappcyl1}{Me = I_+\star X\cup_{\{1\}_+\star X} Y.}
Then by Lemma \ref{ldotps}, 
\beg{meinc}{X=\{0\}_+\star X\r Me.}
is an inclusion, while the projection
$$Me\r \{1\}_+\star X\cup_{\{1\}_+\star X} Y=Y$$
is a homotopy equivalence by Lemma \ref{tdotps1}. 

In effect, 
consider a cell decomposition of \rref{emappcyl1} considered as a relative combinatorial cell spectrum
with respect to the subspectrum $Y$. Then for a cell $a$ of \rref{emappcyl1}, 
we claim that the lowest $j$ such that for 
all $i>j$, $d_i(a)=*$ is the same for $a$ considered as a cell in $I_+\star X$. This is because whenever 
$d_ia\in Y$, there exists an $i^\prime>i$ with $d_{i^\prime}a\notin Y$, such that in $I_+\star X$, $d_ia$ and
$d_{i^\prime}a$ have the same projection to $X$. (This statement depends on the fact that $I_+\star X$ is
attached to $Y$ at the $1$ coordinate.)

Therefore, by the description of $I_+\star?$ given below Lemma \ref{ldotps}, $I_+\star?$ commutes with
the pushout \rref{emappcyl1}, and therefore we can apply Lemma \ref{tdotps1}.

Thus, $e$ may be replaced with the injective morphism
\rref{meinc}.

Now when $e$ is injective, it is contained in an anodyne extension by Proposition \ref{panod}, to which $e$
can be extended by the assumption that $Z$ is Kan fibrant.

To prove that \rref{elocale} is injective, again, without loss of generality, we may assume
that $e$ is injective. Now form the homotopy pushout
$$Pe=Y\cup_{\{1\}_+\star X} \mathscr{I}_+\star X\cup_{\{1^\prime\}_+\star X} Y.$$
Here $\mathscr{I}$ is the simplicial set obtained by attaching two copies of $I$ by $0$; the two other vertices
are denoted by $1,1^\prime$.
We have an inclusion
$$\phi:Pe\subseteq \mathscr{I}_+\star Y,$$
which is moreover an equivalence since $e$ was an equivalence. Now two morphisms $f,g:Y\r Z$ which
are homotopic when composed with $e$ are the same thing as a morphism $\psi:Pe\r Z$. Since $\phi$ is contained
in an anodyne extension by Proposition \ref{panod}, $\psi$ extends to $\mathscr{I}_+\star Y$, which means that
$f$ and $g$ are homotopic, which is what we wanted to prove.
\end{proof}

\begin{corollary}\label{crg}
The category of combinatorial spectra is right Cartan-Eilenberg in the sense of \cite{guillen}
with respect to homotopy equivalences, weak equivalences and Kan fibrant combinatorial spectra.
\end{corollary}

\vspace{5mm}
\section{Cosimplicial realization}\label{scr}

Consider the category $\Delta\dash \mathscr{C}$ of cosimplicial objects in $\mathscr{C}$ where
$\mathscr{C}$ is either the category $\Delta^{Op}\dash Set_\bullet$ of based simplicial sets or
the category $\mathscr{S}$ of combinatorial spectra. We shall construct geometric realization
functors
\beg{ecosreal}{|?|:\Delta\dash \mathscr{C}\r\mathscr{C}.
}
First, let us consider the case of $\mathscr{C}=\Delta^{Op}\dash Set_\bullet$. For the terms
of a cosimplicial
based simplicial set $X$, we will use the notation $X^m_n$ where $m$ is the cosimplicial and $n$ is
the simplicial coordinate. We define $|X|$ as the equalizer in the category of the diagram
\beg{ecosreal1}{\prod_{m\in \N_0}F(\Delta_{m+}, X^m_?)\begin{array}{c}\r \\[-1ex] \r\end{array}
\prod_{\phi\in\Delta^{in}(k,m)} F(\Delta_{m+}, X^k_?)
}
where $\Delta^{in}$ denotes the subcategory of $\Delta$ consisting of injective morphisms only, and
the two arrows \rref{ecosreal1} correspond to applying the morphism in $\Delta^{in}(k,m)$ either
to $\Delta_{m+}$ or to $X^{k}_?$. 

Notice that we are ``ignoring the degeneracies" in \rref{ecosreal1}. Clearly, this realization
of cosimplicial based simplicial sets is a functor. By a {\em weak equivalence} of cosimplicial based
simplicial sets $f:X\r Y$ we mean a morphism such that for every $m$, $f^m_?:X^m_?\r Y^M_?$
is a weak equivalence. A cosimplicial based simplicial set $X$ is {\em term-wise Kan fibrant} if
each $X^m_?$ is Kan fibrant.

\begin{lemma}\label{ldesc1}
The functor $|?|$ preserves weak equivalences on levelwise Kan fibrant based simplicial sets. 
\end{lemma}

\begin{proof}
This follows from the fact that the canonical morphism $X\r F(\Delta_{m,+},X)$ is a weak equivalence when $X$
is Kan fibrant, that $F(?,X)$ turns Kan cofibrations into Kan fibrations, that pullbacks of simplicial sets 
along fibrations preserve weak equivalences, as do directed (inverse) limits of fibrations.
\end{proof}

\vspace{3mm}
\noindent
{\bf Comment:} In \cite{bf}, both co-faces and co-degeneracies are used in defining the cosimplicial realization of
simplicial sets. While that construction seems more natural, it only preserves weak equivalences
on cosimplicial simplicial sets which are fibrant in a stronger sense. Basically, one must require
that the morphism from a given cosimplicial stage to the pullback of all co-degeneracies from the lower
stages is a fibration. We do not know if this has been checked rigorously for the case of Godement
resolutions. 

The difficulty is (Eckmann-Hilton) dual to a similar difficulty with the totalization (=geometric
realization) of simplicial spaces. There, it is important that both faces and degeneracies be used in the 
realization, since only then does one have Milnor's theorem on preserving products (at least as long
as we are in the compactly generated category), which, in turn, is needed when discussing algebraic 
structures (for example, when using the iterated bar construction to construct Eilenberg-MacLane spaces).

In the case of simplicial spaces, one may use Lillig's theorem \cite{lillig} to conclude that individual 
degeneracies being cofibrations is enough to control the homotopy type of the totalization. As far as we know,
a dual of Lillig's theorem for cosimplicial simplicial sets is not known. Nor is this, however, as urgent a problem
as in the case of simplicial spaces: while in the present paper the emphasis is not on further algebraic structures,
cosimplicial realization, even without co-degeneracies, preserves limits (in particular, products)
by the commutation of limits.

\vspace{3mm}
We now apply the definition \rref{ecosreal1} to the case where $X$ is a cosimplicial combinatorial spectrum, thus 
giving a definition of \rref{ecosreal} to $\mathscr{C}=\mathscr{S}$. 
First, observe that a morphism of combinatorial spectra $f:Z\r T$ is a Kan fibration 
if and only if each $f_n:Z_n\r T_n$ is a Kan fibration. Also
note that $\Omega$ preserves Kan fibrations. Finally, directed colimits also preserve Kan fibrations of simplicial
sets (and hence combinatorial spectra). 

\vspace{5mm}

\section{Sheaves}\label{ssh}

We begin with sheaves of sets. We follow \cite{sites} as a reference here. A {\em site} is a category
$\mathscr{C}$
together with a class of sets of morphisms with the same target (called {\em coverings}) which satisfy the usual 
axioms (an isomorphism is a covering, coverings are transitive, and stable under pullback), see \cite{sites}, 
Section 6. A {\em presheaf valued in a category $\mathscr{A}$} is a functor
$\mathscr{F}:\mathscr{C}^{Op}\r\mathscr{A}$. The images of objects (resp. morphisms)
under $\mathscr{F}$ are called {\em sections} (resp. {\em restrictions}).
A {\em sheaf} is a presheaf $\mathscr{F}$ such that
for every covering $\{X_i\r X\}$, the diagram
$$\mathscr{F}(X)\r \prod_i \mathscr{F}(X_i)\begin{array}{c}\r\\[-2ex]\r\end{array}
\prod_{i,j} \mathscr{F}(X_i\times_{X}X_j)$$
where the maps are restrictions is an equalizer. A site is said to have {\em subcanonical topology}
it all representable presheaves are sheaves. The representable (pre)sheaf associated with an object $x$ of
a site will be denoted by $\underline{x}$. 

Morphisms of presheaves are natural transformations,
and sheaves are a full subcategory. The categories of $\mathscr{A}$-valued presheaves and sheaves
on a site $\mathscr{C}$ will be denoted by $pSh_\mathscr{A}(\mathscr{C})$
resp. $Sh_\mathscr{A}(\mathscr{C})$. If the subscript is omitted, we understand $\mathscr{A}=Set$.
The category of sheaves of sets on a site $\mathscr{C}$ is called its {\em topos}.
If $\mathscr{A}$ is a category of universal algebras, then the category of $\mathscr{A}$-valued sheaves is
equivalent to the category of the same type of universal algebras in the topos.
In this case, the forgetful functor from sheaves valued in $\mathscr{A}$ to presheaves valued in 
$\mathscr{A}$ has a left adjoint called {\em sheafification} (\cite{sites}, Section 10). 

A {\em morphism of topoi} $f:Sh(\mathscr{C})\r Sh(\mathscr{D})$ consists of a functor
$$f^{-1}:Sh(\mathscr{D})\r Sh(\mathscr{C})$$
which has a right adjoint $f_*$, and is left exact, i.e preserves finite limits. A {\em point} is 
a morphism of a topos into the topos of sets (i.e. the category of sheaves on $*$). The set $f^{-1}\mathscr{F}$
for a sheaf $\mathscr{F}$ where $f$ is a point is called a {\em stalk}. Points
in this sense can be also characterized in terms of the site $\mathscr{C}$ directly (\cite{sites},
Lemma 31.7). We say that a site has {\em enough points} if a morphism of sheaves is an isomorphism
whenever it is an isomorphism on stalks. In this paper, we will work with sites $\mathscr{C}$ satisfying
the following assumption:
\beg{ass1}{\parbox{3.5in}{The site $\mathscr{C}$ is small (i.e. is a set) and has enough points.}\tag{A1}
}
We call a morphism of sheaves {\em injective} resp. {\em surjective (or onto)} if it is injective resp. surjective
on stalks. Being injective is equivalent to being injective on sections. By \cite{sites}, Lemma 28.5, 
we may make without loss of generality (i.e. by replacing the topos with an isomorphic topos) the following
assumption:
\beg{ess2}{\parbox{3.5in}{The site $\mathscr{C}$ has subcanonical topology, and 
subsheaves of representable sheaves are representable.}\tag{A2}}
Since points are characterized in terms of the topos, in particular, we may attain Assumption (A1) without violating
Assumption (A2).

Finally, Lemma 28.5 of \cite{sites} says that we can replace, equivalently from
the point of view of the structures we study, a (small) site with a given set of sheaves $S$
by a (small) site satisfying 
Assumption (A2) where every element of $S$ becomes representable. Choosing $S$ to be the set of all
quotients (i.e. images of surjective morphisms) of representable sheaves, we may assume $\mathscr{C}$
additionally satisfies the following assumption:
\beg{ess3}{\parbox{3.5in}{For every sheaf $\mathscr{G}$ on $\mathscr{C}$, every point $p$ and
every element $t\in p^{-1}(\mathscr{G})$, there exists an object $u$ of $\mathscr{C}$
and an injective morphism $\underline{u}\r\mathscr{G}$ such that $p$ is in $u$
and $t$ lifts to $\mathscr{G}(u)$.}\tag{A3}}

\begin{lemma}\label{lsglue}
Consider a site $\mathscr{C}$ satisfying Assumptions (A1) and (A2). Suppose we have injective morphisms
of sheaves
\beg{ealphabeta}{\alpha:\underline{x}\r\mathscr{G},\;\beta:\mathscr{F}\r\mathscr{G}}
for some $x\in Obj\mathscr{C}$.
Then there exists a monomorphism $\iota:y\r x$ in $\mathscr{C}$ and a pushout diagram
\beg{ealphabetap}{\diagram
\underline{y}\rto^{\underline{\iota}}\dto_{\phi} &
\underline{x}\dto^{\alpha}\\
\mathscr{F}\rto_\beta\ &\mathscr{F}^\prime
\enddiagram
}
such that the induced morphism
$$\beta^\prime:\mathscr{F}^\prime\r\mathscr{G}$$
is injective.
\end{lemma}

\begin{proof}
We define the diagram \rref{ealphabetap} as the pullback of $\alpha$ and $\beta$. The pullback sheaf is
representable by Assumption (A2). The top row is representable, and $\iota$ is a monomorphism by the
Yoneda lemma. Finally, the reason $\beta^\prime$ is also injective is that it is true after applying
$f^{-1}$ (which is an exact functor) for any point $f$, which is enough by Assumption (A1).
\end{proof}

\vspace{3mm}
To discuss homotopy theory, we begin with sheaves of simplicial sets (simplicial sheaves). Simplicial sets are a
universal algebra, so the category is determined by the topos. By a {\em local equivalence} of simplicial
sheaves (i.e. objects of
$\Delta^{Op}\dash Sh(\mathscr{C})$), we mean a morphism which is
a weak equivalence on stalks. 



\vspace{3mm}
For a monomorphism $\iota:y\r x$ in $\mathscr{C}$,
denote by $\Delta_{n}^{\iota}$ the pushout in the category of simplicial sheaves
\beg{edeltaiota}{\diagram
\Delta_n^\circ\times \underline{y}\rto^{\Delta_n^\circ\times \iota}\dto_{\subseteq
\times Id} & \Delta_n^\circ\times \underline{x}\dto\\
\Delta_n\times\underline{y}\rto & \Delta_{n}^{\iota}.
\enddiagram
}
Here the product of a (simplicial) set with a sheaf of sets is done section-wise. (It also commutes
with taking stalks.)

An injective morphism of sheaves $f:\mathscr{F}\r\mathscr{G}$ is called a
{\em relative cell sheaf} if there exists an ordinal
$\alpha$ and sheaves $\mathscr{G}_\gamma$ for ordinals $\gamma<\alpha$, where
$\mathscr{G}_0=\mathscr{F}$, $\mathscr{G}_\alpha=\mathscr{G}$, for a limit ordinal $\beta$,
$\mathscr{G}_\beta$ is the colimit of $\mathscr{G}_\gamma$, $\gamma<\beta$, such that for any ordinal
$\beta<\alpha$, we have a monomorphism $\iota_\beta:y_\beta\r x_\beta$ in $\mathscr{C}$ 
and a pushout of sheaves of the form
\beg{edefcellsh}{\diagram
\Delta_{n_\beta}^{\iota_\beta}\rto\dto & \mathscr{G}_\beta\dto\\
\Delta_{n_\beta}\times \underline{x}_\beta\rto& \mathscr{G}_{\beta+1}.
\enddiagram
}
An anodyne extension of sheaves is defined the same way, except $\Delta_{n_\beta}^{\iota_\beta}$
is replaced by $E_{n_\beta, k_\beta}^{\iota_\beta}$, which is defined by the pushout diagram
$$
\diagram
V_{n,k}\times \underline{y}\rto^{V_{n,k}\times \iota}\dto_{\subseteq
\times Id} & V_{n,k}\times \underline{x}\dto\\
\Delta_n\times\underline{y}\rto & E_{n,k}^{\iota}.
\enddiagram
$$
Explicitly, an injective morphism of sheaves $f:\mathscr{F}\r\mathscr{G}$ is called an
{\em anodyne extension} if there exists an ordinal
$\alpha$ and sheaves $\mathscr{G}_\gamma$ for ordinals $\gamma<\alpha$, where
$\mathscr{G}_0=\mathscr{F}$, $\mathscr{G}_\alpha=\mathscr{G}$, for a limit ordinal $\beta$,
$\mathscr{G}_\beta$ is the colimit of $\mathscr{G}_\gamma$, $\gamma<\beta$, such that for any ordinal
$\beta<\alpha$, we have a monomorphism $\iota_\beta:y_\beta\r x_\beta$ in $\mathscr{C}$ 
and a pushout of sheaves of the form
\beg{edefcellsh}{\diagram
E_{n_\beta, k_\beta}^{\iota_\beta}\rto\dto & \mathscr{G}_\beta\dto\\
\Delta_{n_\beta}\times \underline{x}_\beta\rto& \mathscr{G}_{\beta+1}.
\enddiagram
}

\begin{lemma}\label{lconstrsh}
Every injective morphism of simplicial sheaves over a site $\mathscr{C}$ which satisfies Assumptions
(A1), (A2) and (A3) is a relative cell sheaf.
\end{lemma}

\begin{proof}
Consider an injective morphism of sheaves of simplicial 
sets $\mathscr{F}\r\mathscr{G}$. Put $\mathscr{G}_0=\mathscr{F}$.
We will construct, inductively, sheaves $\mathscr{G}_\beta$ as in the definition of a relative cell sheaf. For $\beta$
a limit ordinal, just take the colimit of $\mathscr{G}_\gamma$ over $\gamma<\beta$. For any ordinal $\beta$, 
we will have an injective morphism $\mathscr{G}_\beta\r \mathscr{G}$. If it is onto, we are done. Otherwise,
by Assumption (A1), there exists a point $p$ and an element $t\in p^{-1}(\mathscr{G})$ which is not in 
$\mathscr{G}_\beta$, but whose faces are. 
By Assumption (A3), there exists an object $u$ containing $p$ and an injective morphism
$\alpha:\underline{u}\r \mathscr{G}$ such that $t$ lifts to $\mathscr{G}(u)$ via $\alpha$. 

Now we are in the situation of Lemma \ref{lsglue}, with $\mathscr{F}$ replaced by 
the appropriate term of $\mathscr{G}_\beta$. 
Let $\mathscr{G}_{\beta+1}$ be attaching one non-degenerate simplex in the appropriate
dimension by the pushout \rref{ealphabetap}. 

The process is guaranteed to end by the smallness of $\mathscr{C}$.
\end{proof}

\vspace{3mm}

\begin{lemma}\label{lanodsh}
An anodyne extension of sheaves over a site $\mathscr{C}$ which satisfies Assumptions
(A1) and (A2) is a local equivalence.
\end{lemma}

\begin{proof}
Use the definition, performing the construction on the presheaf level 
first, and then sheafifying.
\end{proof}

\begin{lemma}\label{lanodextsh}
For an injective local equivalence of sheaves $f:\mathscr{F}\r\mathscr{G}$
over a site $\mathscr{C}$ which satisfies Assumptions
(A1), (A2) and (A3), there exists a morphism of sheaves
$g:\mathscr{G}\r\mathscr{H}$ such that $gf$ is an anodyne extension.
\end{lemma}

\begin{proof}
We shall imitate the proof of Lemma \ref{lconstrsh}, expressing 
$f$ as a relative cell sheaf. Using the notation in the definition, we will 
construct, by induction, morphisms 
\beg{egbetata}{g_\beta:\mathscr{G}_\beta\r\mathscr{H}_\beta,}
\beg{egbetata1}{\mathscr{G}_\beta\subseteq \mathscr{G},}
such that $g_\beta f$ is an anodyne extension. For $\beta$ a limit ordinal, we can just take the directed
direct limit, so let us assume \rref{egbetata} was constructed for a given $\beta$. 
Again, if \rref{egbetata1} is onto, we are done. Otherwise, by Assumption (A1), there exists a point $p$ and
an element $t\in p^{-1}(\mathscr{G}) $ which is not in $\mathscr{G}_\beta$ but whose faces are. By Assumption
(A3), there exists an object $u$ containing $p$ and an injective morphism $\alpha:\underline{u}\r \mathscr{G}$
such that $t$ lifts to $\widetilde{t}\in
\mathscr{G}(u)$ via $\alpha$ and whose faces are in $\mathscr{G}_\beta$. Furthermore, since
$f$ is a local equivalence, we may assume, upon replacing $u$ with another object $x_\beta\subseteq u$, 
that $\widetilde{t}$
lifts to the geometric realization of $\mathscr{H}_\beta(x_\beta)$ up to homotopy. Now considering the subobject 
$y_\beta$
as in Lemma \ref{lsglue}, 
we can assume without loss of generality that the restriction
\beg{egbeta1}{\mathscr{H}_\beta(x_\beta)\r\mathscr{H}_\beta(y_\beta)
}
is a Kan fibration. To this end, attach, in $\omega$ steps, each time all non-isomorphic pushouts of the form
$$\diagram
E_{n,k}^{\iota_\beta}\dto_\subset \rto & ?\\
\Delta_n\times \underline{x}_\beta &
\enddiagram
$$
Despite the fact that this is not exactly the same process as the canonical factorization into an anodyne extension
and Kan fibration of simplicial sets (since sheafification is performed each time), the small object  nevertheless
applies, so taking the colimit over the $\omega$ steps replaces $\mathscr{G}_\beta$ with a sheaf where
\rref{egbeta1} is a Kan fibration. 

Now consider the pushout defining $\mathscr{G}_{\beta+1}$ in \rref{edefcellsh}. By our assumption, 
the composition 
$$\Delta^{\iota_\beta}_{n_\beta}\r\mathscr{G}_\beta\r\mathscr{G}$$
extends to
$$\Delta_{n_\beta}\times\underline{x}_\beta\r\mathscr{G},$$
and moreover by our assumptions,
$$\Delta^{\iota_\beta}_{n_\beta}\r\mathscr{G}_\beta\r\mathscr{H}_\beta$$
extends to 
\beg{edeltaextet}{\Delta_{n_\beta}\times \underline{x}_\beta\r \mathscr{H}_\beta}
by adjunction and the assumption that \rref{egbeta1} is a Kan fibration. (Caution: We do not know that
$\mathscr{F}(x)\r\mathscr{H}_\beta(x)$ is a weak equivalence! However, it does not matter, since
by assumption, the homotopy lifting problem can be solved in $|\mathscr{F}(x)|$, and in a Kan fibration,
a homotopy lifting problem which has a solution upon geometric realization has a solution.)

But since we have \rref{edeltaextet}, we may extend the pushout \rref{edefcellsh} to a pushout of the
form
\beg{ehext}{\diagram
Q_{n_\beta}^{\iota_\beta}\rto\dto &\mathscr{H}_\beta\dto\\
\Delta_{n_\beta}\times I\times \underline{x}_\beta \rto &\mathscr{H}_{\beta+1}
\enddiagram}
where $Q_{n}^\iota$ for a monomorphism $\iota:y\r x$ in $\mathscr{C}$ is defined as a pushout
$$\diagram
G_n\times \underline{y}\rto\dto &G_n\times \underline{x}\dto\\
\Delta_n\times I\times \underline{y}\rto & Q_n^\iota& 
\enddiagram
$$
where
$$G_n=\Delta_n^\circ\times I\cup \Delta_n\times\{0\}.$$
But \rref{ehext} is an anodyne extension. Again, the process must eventually terminate for set-theoretical reasons.
\end{proof}

\vspace{3mm}
A {\em strong homotopy} of simplicial sheaves on $\mathscr{C}$ is a morphism of the form
$$h:\mathscr{F}\times I\r \mathscr{G}.$$
Multiplication by $I$ is performed section-wise, and we sheafify the result. 
$h$ is also called a strong homotopy between the restriction of $h$ to $\mathscr{F}\times \{0\}$ 
and $\mathscr{F}\times\{1\}$. We may now consider the smallest equivalence relation on morphisms
of simplicial sheaves which contains strong homotopy. This is obviously a congruence, and the 
quotient category is called the {\em strong homotopy category of simplicial sheaves}. 

\begin{lemma}\label{lshgh}
A strong homotopy equivalence of sheaves of simplicial sets is an equivalence on sections, and hence
a local equivalence.
\end{lemma}

\begin{proof}
For a strong homotopy of simplicial sheaves
$$I\times \mathscr{F}\r\mathscr{G}$$
and every $u\in Obj(\mathscr{C})$, we obtain, by definition, a simplicial homotopy on sections
$$I\times \mathscr{F}(u)\r\mathscr{G}(u).$$
Therefore, a strong homotopy equivalence of sheaves gives a simplicial homotopy equivalence after applying
sections, hence a local equivalence.
\end{proof}

\vspace{3mm}
\begin{theorem}\label{tcegodement}
Under the Assumptions (A1), (A2) and (A3), and assuming also that $\mathscr{C}$ has finite cohomological
dimension, the category $\Delta^{Op}\dash Sh(\mathscr{C})$
of simplicial sheaves on $\mathscr{C}$ is right Cartan-Eilenberg with respect to
strong homotopy equivalence, local equivalence and cosimplicial Godement resolutions.
\end{theorem}

\begin{proof}
First note the following:
\beg{etceclaim}{\parbox{3.5in}{Cosimplicial Godement resolutions have the property
that restrictions under monomorphisms in $\mathscr{C}$ are Kan fibrations.}
}
Thus, it suffices to show that for a local equivalence $e:\mathscr{F}\r\mathscr{G}$ and a
cosimplicial Godement resolution $\mathscr{X}$, we have a bijection
\beg{egodbij}{[e,\mathscr{X}]:[\mathscr{G},\mathscr{X}]\r [\mathscr{F},\mathscr{X}]
}
where $[?,?]$ denotes the set of strong homotopy classes.

This is done in the standard way: To prove that \rref{egodbij} is onto, we first replace $\mathscr{G}$
by the mapping cylinder of $e$. Then Lemma \ref{lanodextsh} applies (with $\mathscr{K}=\mathscr{G}\times I$).
Therefore, the mapping cylinder embeds to an anodyne extension, for which the mapping extension 
problem into $\mathscr{X}$ can be solved by \rref{etceclaim}.

To prove that \rref{egodbij} is injective, first replace $e$ by its mapping cylinder (which is
isomorphic to $\mathscr{G}$ in the strong homotopy category) to make $e$ injective. Then build a 
pushout $\mathscr{P}$ of two copies of $e$. Then
the embedding 
\beg{emapg}{\mathscr{P}\subseteq \mathscr{G}\times \mathscr{I}} 
is a local equivalence (where $\mathscr{I}$ is the simplicial set obtained from attaching two copies of $I$ at a point)
and hence it is contained in an anodyne extension. Therefore, we may conclude that for a cosimplicial
Godement resolution $\mathscr{X}$, a morphism $\mathscr{P}\r\mathscr{X}$ extends to $\mathscr{G}\times \mathscr{I}$,
which shows that \rref{egodbij} is injective.

The fact that under our assumptions, the canonical morphism of a sheaf to the cosimplicial realization of its
cosimplicial Godement resolution is a local equivalence follows from \cite{rg}, Theorem 4.14.
\end{proof}

\vspace{3mm}
The case of sheaves of based simplicial sets and combinatorial spectra is now treated analogously, with $I\times ?$
replaced by $I_+\wedge ?$, resp. $I_+\star ?$. Let us discuss the case of combinatorial spectra in more detail.
A sheaf of combinatorial spectra is a functor from a site $\mathscr{C}$ into the category of combinatorial spectra which 
satisfies the sheaf limit condition in the category of combinatorial spectra. Since combinatorial spectra are a 
coreflexive subcategory of $\Delta^{Op}_{st}\dash Set_\bullet$ (which is a category of universal algebras), 
sheafification of a presheaf of combinatorial spectra can be constructed by sheafifying the corresponding 
presheaf valued in
$\Delta^{Op}_{st}\dash Set_\bullet$.
Since the condition of being a combinatorial spectrum cannot be called a universal algebra condition, we do not know 
if the category of sheaves of combinatorial spectra is independent of the choice of sites defining the same topos. 
Nevertheless, we shall assume that our site $\mathscr{C}$ satisfies Assumptions (A1), (A2) and (A3).
By a {\em local equivalence} of 
sheaves of combinatorial spectra, we mean a morphism which is
a weak equivalence on stalks. 



\vspace{3mm}
An injective morphism of sheaves of combinatorial spectra $f:\mathscr{F}\r\mathscr{G}$ is called a
{\em relative cell sheaf} if we have sheaves $\mathscr{G}_\gamma$ for ordinals $\gamma<\alpha$,
$\mathscr{G}_0=\mathscr{F}$, $\mathscr{G}_\alpha=\mathscr{G}$, for a limit ordinal $\beta$,
$\mathscr{G}_\beta$ is the colimit of $\mathscr{G}_\gamma$, $\gamma<\beta$, and for any ordinal
$\beta<\alpha$, we have an injective morphism $\iota_\beta:y_\beta\r x_\beta$ in $\mathscr{C}$ 
and a pushout of sheaves of the form
\beg{sedefcellsh}{\diagram
\Sigma^{\infty-\ell_\beta}\Delta_{n_\beta+}^{\iota_\beta}\rto\dto & \mathscr{G}_\beta\dto\\
\Sigma^{\infty-\ell_\beta}(\Delta_{n_\beta}\times \underline{x}_\beta)_+\rto& \mathscr{G}_{\beta+1}.
\enddiagram
}
Here $\Sigma^{\infty-\ell}$ of a sheaf of based simplicial set is constructed by taking shift suspension
spectra section-wise and then sheafifying. This is again left adjoint to taking $\Omega^{\infty-\ell}$ 
section-wise.

An {\em anodyne extension} of sheaves is defined the same way, except $\Delta_{n_\beta}^{\iota_\beta}$
is replaced by $E_{n_\beta, k_\beta}^{\iota_\beta}$. 

Analogously to Lemma \ref{lconstrsh}, we have

\begin{lemma}\label{slconstrsh}
Every injective morphism of sheaves of combinatorial spectra over a site $\mathscr{C}$ which satisfies Assumptions
(A1), (A2) and (A3) is a relative cell sheaf.
\end{lemma}

\begin{proof}
A verbatim repeat of the proof of Lemma \ref{lconstrsh}, with simplicial sets replaced by combinatorial
spectra.
\end{proof}

\vspace{3mm}
Analogously to Lemma \ref{lanodsh}, we have

\begin{lemma}\label{slanodsh}
An anodyne extension of sheaves of combinatorial spectra over a site $\mathscr{C}$ which satisfies Assumptions
(A1) and (A2) is a local equivalence.
\end{lemma}

\begin{proof}
Again, the proof is the same as for Lemma \ref{lanodsh}.
\end{proof}

Analogously to Lemma \ref{lanodextsh}, we have

\begin{lemma}\label{slanodextsh}
For an injective weak equivalence of sheaves of combinatorial spectra $f:\mathscr{F}\r\mathscr{G}$
over a site $\mathscr{C}$ which satisfies Assumptions
(A1), (A2) and (A3), there exists a morphism of sheaves
$g:\mathscr{G}\r\mathscr{H}$ such that $gf$ is an anodyne extension.
\end{lemma}

\begin{proof}
This time, the proof of Lemma \ref{lanodextsh} requires some modification, to account for shift desuspensions. 
To clarify, we write it out. The beginning is the same:
By Lemma \ref{slconstrsh}, $f$ is a relative cell sheaf. Using the notation in the definition, we will 
construct, by induction, morphisms 
\beg{segbeta}{g_\beta:\mathscr{G}_\beta\r\mathscr{H}_\beta}
such that $g_\beta f$ is an anodyne extension. For $\beta$ a limit ordinal, we can just take the directed
direct limit, so let us assume \rref{segbeta} has been constructed for a given $\beta$. 
Again, we may assume using the fact that $f$ is a local equivalence that the embedding 
$\iota_\beta:y_\beta\subseteq x_\beta$ is chosen is such a way that
the boundary of the new cell can be extended to the geometric realization of
$\mathscr{H}_\beta$ up to homotopy.
We now claim that
we can assume without loss of generality that the restriction
\beg{segbeta1}{\mathscr{H}_\beta(x_\beta)\r\mathscr{H}_\beta(y_\beta)
}
is a Kan fibration of combinatorial spectra (i.e. a level-wise Kan fibration). 
To this end, attach, in $\omega$ steps, each time all non-isomorphic pushouts of the form
$$\diagram
\Sigma^{\infty-\ell}E_{n,k}^{\iota_\beta+}\dto_\subset \rto & ?\\
\Sigma^{\infty-\ell}(\Delta_n\times \underline{x}_\beta)_+ &
\enddiagram
$$
Despite the fact that this is not exactly the same process as the canonical factorization into an anodyne extension
and Kan fibration of combinatorial spectra (since sheafification is performed at each time), the small object  nevertheless
applies, so taking the colimit over the $\omega$ steps replaces $\mathscr{G}_\beta$ with a sheaf where
\rref{egbeta1} is a Kan fibration of combinatorial spectra. 

Now consider the pushout defining $\mathscr{G}_{\beta+1}$ in \rref{sedefcellsh}. By our assumption, 
the composition 
$$\Sigma^{\infty-\ell_\beta}\Delta^{\iota_\beta}_{n_\beta+}\r\mathscr{G}_\beta\r\mathscr{G}$$
extends to
$$\Sigma^{\infty-\ell_\beta}(\Delta_{n_\beta}\times\underline{x}_\beta)_+,$$
which by our assumption on $\iota_\beta$ means
$$\Sigma^{\infty-\ell_\beta}\Delta^{\iota_\beta}_{n_\beta+}\r\mathscr{G}_\beta\r\mathscr{H}_\beta$$
extends to 
\beg{sedeltaext}{\Sigma^{\infty-\ell_\beta}(\Delta_{n_\beta}\times \underline{x}_\beta)_+\r \mathscr{H}_\beta}
by adjunction and the assumption that \rref{segbeta1} is a Kan fibration of
combinatorial spectra. (Caution: We do not know that
$\mathscr{F}(x)\r\mathscr{H}_\beta(x)$ is a weak equivalence! However, it does not matter, since
by assumption, the homotopy lifting problem can be solved in $|\mathscr{F}(x)|$, and in a Kan fibration
of combinatorial spectra, just as of simplicial sets,
a homotopy lifting problem which has a solution upon geometric realization has a solution.)

But since we have \rref{sedeltaext}, we may extend the pushout \rref{sedefcellsh} to a pushout of the
form
\beg{sehext}{\diagram
\Sigma^{\infty-\ell_\beta}Q_{n_\beta+}^{\iota_\beta}\rto\dto &\mathscr{H}_\beta\dto\\
\Delta_{n_\beta}\times I\times \underline{x}_\beta \rto &\mathscr{H}_{\beta+1}
\enddiagram}
where $Q_{n}^\iota$ is as in the proof of Lemma \ref{lanodextsh}.
\end{proof}

\vspace{3mm}
A {\em strong homotopy} of sheaves of combinatorial spectra on $\mathscr{C}$ now is a morphism of the form
$$h:I_+\star\mathscr{F}\r \mathscr{G}.$$
Again, $I_+\star?$ is performed section-wise, and we sheafify the result. 
$h$ is also called a strong homotopy between the restriction of $h$ to $\{0\}_+\star\mathscr{F}$ 
and $\{1\}_+\star\mathscr{F}$. We may now consider the smallest equivalence relation on morphisms
of simplicial sheaves which contains strong homotopy. This is obviously a congruence, and the 
quotient category is called the {\em strong homotopy category of sheaves of combinatorial spectra}. 

\begin{lemma}\label{slshgh}
A strong homotopy equivalence of sheaves of combinatorial spectra is an equivalence on sections,
and hence a local equivalence.
\end{lemma}

\begin{proof}
Analogous to the proof of Lemma \ref{lshgh}.
\end{proof}

\vspace{3mm}
Finally, the analogue of Theorem \ref{tcegodement} is

\vspace{3mm}
\begin{theorem}\label{stcegodement}
Under the Assumptions (A1), (A2) and (A3), and assuming also that $\mathscr{C}$ has finite cohomological
dimension, the category $Sh_\mathscr{S}(\mathscr{C})$
of sheaves of combinatorial spectra on $\mathscr{C}$ is right Cartan-Eilenberg with respect to
strong homotopy equivalence, local equivalence and cosimplicial Godement resolutions.
\end{theorem}

\begin{proof}
Again, we follow the proof of Theorem \ref{tcegodement}, but it
requires modification due to peculiarities of the smash product, 
so we write it out. Again, the beginning is the same.
First note the following:
\beg{setceclaim}{\parbox{3.5in}{Cosimplicial Godement resolutions have the property
that restrictions under monomorphisms in $\mathscr{C}$ are Kan fibrations of combinatorial spectra.}
}
Thus, it suffices to show that for a local equivalence of
combinatorial spectra $e:\mathscr{F}\r\mathscr{G}$ and a
cosimplicial Godement resolution $\mathscr{X}$, we have a bijection
\beg{segodbij}{[e,\mathscr{X}]:[\mathscr{G},\mathscr{X}]\r [\mathscr{F},\mathscr{X}]
}
where $[?,?]$ denotes the set of strong homotopy classes.

This is done in the standard way: To prove that \rref{segodbij} is onto, we first replace $\mathscr{G}$
by the mapping cylinder of $e$. The Lemma \ref{lanodextsh} applies (with $\mathscr{K}=I_+\star\mathscr{G}$).
Therefore, the mapping cylinder embeds to an anodyne extension, for which the mapping extension 
problem into $\mathscr{X}$ can be solved by \rref{setceclaim}.

To prove that \rref{segodbij} is injective, first replace $e$ by its mapping cylinder (which is
isomorphic to $\mathscr{G}$ in the strong homotopy category) to make $e$ injective. Then build a 
pushout $\mathscr{P}$ of two copies of $e$. We note that we do not know that
the embedding 
\beg{semapg}{\mathscr{I}_+\star\mathscr{P}\subseteq \mathscr{G}} 
is a local equivalence (where $\mathscr{I}$ is again the simplicial set obtained from attaching two copies of $I$ at a point). 
Thus, we know \rref{semapg} is contained in an anodyne extension of 
$\mathscr{P}$. Therefore, we may still conclude that for a cosimplicial
Godement resolution $\mathscr{X}$, a morphism $\mathscr{P}\r\mathscr{X}$ extends to $\mathscr{I}_+\star\mathscr{G}$,
which shows that \rref{segodbij} is injective.

The fact that under our assumptions, the canonical morphism of a sheaf to the cosimplicial realization of its
cosimplicial Godement resolution is a local equivalence again follows from \cite{rg}, Theorem 4.14.
\end{proof}

\vspace{5mm}
\section{Comparison with Thomason}\label{thomason}

Thomason \cite{thomason} considers a concept of {\em fibrant simplicial $\Omega$-spectra} (which he attributes to 
Bousfield and Friedlander \cite{bf}) which are
sequences
$$Z_n,n\in \N_0$$
(alternatively, $n\in \Z$) of based simplicial sets, together with structure morphisms
\beg{ethomstr}{S^1\wedge Z_n\r Z_{n+1},}
such that the simplicial sets $Z_n$ are Kan fibrant, and the adjoints of the structure maps \rref{ethomstr} are
weak equivalences:
$$\diagram
Z_n\rto^(.3)\sim & F(S^1,Z_{n+1}).
\enddiagram$$
Morphisms $(Z_n)\r (T_n)$, as usual, are sequences of morphisms of based simplicial sets which commute with
the structure morphisms \rref{ethomstr}. 

Thomason \cite{thomason} 
observes that his category of fibrant simplicial spectra has products and directed colimits, which 
means that on sites with enough points, 
local equivalences may be defined as morphisms
which induce weak equivalences on stalks. Global equivalence are defined as morphisms
which induce equivalences on sections. Additionally, on sites which 
have finite cohomological dimension,
where one can use 
cosimplicial Godement 
resolutions, one can treat generalized sheaf cohomology completely on the level of presheaves.
The reason is that stalks, again, can be calculated on the level of presheaves and  on the other hand, the products
of skyscraper sheaves which occur in cosimplicial Godement resolutions are sheaves in any subcanonical topology.
It was proved in \cite{rg} that under these assumptions, the category of Thomason presheaves of
fibrant simplicial spectra is right Cartan-Eilenberg with respect to equivalences on sections, local equivalences
and Godement resolutions.

An actual general theory of sheaves of Thomason's fibrant simplicial spectra, on the other hand, 
does not seem meaningful, as more types of colimits are necessary, for example, to have sheafification.

\vspace{3mm}
By Proposition \ref{p1}, we have a canonical natural transformation
\beg{enatsmash}{S^1\wedge X=\Sigma(S^0)\wedge X\r \Sigma(S^0\wedge X)=\Sigma X.
}
This means that Kan fibrant combinatorial spectra are canonically a full subcategory of Thomason fibrant simplicial
spectra, and local equivalences coincide with equivalences on sections. 
By Lemma \ref{lsingsp} (2), every combinatorial spectrum can
be functorially replaced by a Kan fibrant one, so doing this section-wise gives a functor from sheaves of
combinatorial spectra to presheaves of Kan simplicial spectra which preserves local equivalences,
as well as equivalences on sections. 

There is also a functor the other way. In fact, the construction applies to any sequence of based simplicial
sets with connecting morphisms
\rref{ethomstr} without any additional assumptions. We may then apply functorial fibrant replacement. 
The construction is performed in two steps. First, replace $Z_n$ with a sequence of based simplicial
sets $Z^\prime_n$ and connecting maps \rref{ethomstr} with $Z_n$ replaced by $Z^\prime_n$ which are
injective. We define inductively 
$$Z^\prime_0=Z_0,$$
and assuming we already have a morphism
$$Z^\prime_n\r Z_n,$$
we let $Z^\prime_{n+1}$ be the mapping cylinder of the composition
$$S^1\wedge Z^\prime_n\r S^1\wedge Z_n\r Z_{n+1}.$$
Thus, assume without loss of generality $Z^\prime=Z$. In the second step, we make this into a
combinatorial prespectrum as follows: Let
$$Z^{\prime\prime}_0=Z_0.$$
Assuming we have already a morphism of based simplicial sets
$$Z_n\r Z^{\prime\prime}_n,$$
we let $Z^{\prime\prime}_{n+1}$ be the pushout of the canonical diagram
$$\diagram
S^{1}\wedge Z_n\rto \dto &Z_{n+1}\\
\Sigma Z^{\prime\prime}_n. &
\enddiagram
$$
By functoriality, this construction automatically passes to presheaves. We may then spectrify, sheafify and fibrant
replace as we wish. All those functors are left adjoint and commute by commutation of adjoints. 

The advantage of using combinatorial spectra, where we have a fully functional 
theory of sheaves, is that we now also have
functors (and their right derived functors, provided they preserve strong homotopy) which cannot be constructed 
on presheaves alone. Functors of the form $f_*$, $f_!$ for a general morphism of sites $f$ are an example.
(For the functor $f_*$, this was also observed in \cite{rg}, Corollary 4.20, since $f_*$ preserves 
section-wise equivalences.)


\vspace{10mm}
\begin{thebibliography}{99}

\bibitem{adams}  J.F. Adams: 
{\em Stable homotopy and generalised homology,} Chicago Lectures in Mathematics, 
University of Chicago Press, Chicago, IL, 1974

\bibitem{blaz} J.Block, A.Lazarev: Homotopy Theory and Generalized Duality for Spectral Sheaves, {\em
International Mathematics Research Notices} 20 (1996) 983-996

\bibitem{bf} A.K.Bousfield, E.M.Friedlander:
Homotopy theory of $\Gamma$-spaces, spectra, and bisimplicial sets, {\em Geometric applications of homotopy 
theory} (Proc. Conf., Evanston, Ill., 1977), II, pp. 80–130,
Lecture Notes in Math., 658, Springer, Berlin, 1978

\bibitem{bkan}  A.K.Bousfield, D.M.Kan: {\em Homotopy limits, completions and localizations}, 
Lecture Notes in Mathematics, Vol. 304. Springer-Verlag, Berlin-New York, 1972

\bibitem{ksb} K. S. Brown: Abstract homotopy theory and generalised sheaf cohomology, {\em Trans. AMS} 
186 (1973) 419-458

\bibitem{ksb1} K.S. Brown, S.M. Gersten:
Algebraic K-theory as generalized sheaf cohomology,  {\em Algebraic K-theory, I: Higher K-theories} 
(Proc. Conf., Battelle Memorial Inst., Seattle, Wash., 1972), pp. 266-292. Lecture Notes in Math., 
Vol. 341, Springer, Berlin, 1973

\bibitem{ce}  H.Cartan, S. Eilenberg: {\em Homological algebra}, Princeton University Press, Princeton, N. J., 1956

\bibitem{dt} A. Dold, R.Thom: Quasifaserungen und unendliche symmetrische Produkte, (German) 
{\em Ann. of Math.} (2) 67 (1958) 239-281

\bibitem{ekmm}  A.D.Elmendorf, I. Kriz, M.A. Mandell, J.P. May: {\em Rings, modules, and 
algebras in stable homotopy theory. With an appendix by M. Cole}, Mathematical Surveys and Monographs, 
47. American Mathematical Society, Providence, RI, 1997

\bibitem{fkelly}  P.J.Freyd, G.M.Kelly: Categories of continuous functors, I. {\em J. Pure Appl. Algebra} 2 (1972), 
169-191, erratum:  {\em J. Pure Appl. Algebra} 4 (1974), 121

\bibitem{geisser}  T.Geisser: Motivic cohomology, K-theory and topological cyclic homology, 
{\em Handbook of K-theory}, Vol. 1, 2, 193-234, Springer, Berlin, 2005

\bibitem{gillet} H.Gillet: Riemann-Roch theorems for higher algebraic k-theory, 
{\em Bull. Amer. Math. Soc.} (N.S.) 3 (1980), no. 2, 849-852

\bibitem{gillet1} H.Gillet: Riemann-Roch theorems for higher algebraic K-theory, 
{\em Adv. in Math.} 40 (1981), no. 3, 203-289

\bibitem{guillen} F.Guill\'{e}n, V.Navarro, P.Pascual, A.Roig:
A Cartan-Eilenberg approach to homotopical algebra,
{\em J. Pure Appl. Algebra} 214 (2010), no. 2, 140-164

\bibitem{hirsch} P.S.Hirschhorn: {\em Model Categories and their Localizations}, Math. Surves and Monog. 99, AMS, 
2002

\bibitem{hovey} M. Hovey, B. Shipley, J. Smith: Symmetric spectra, {\em J. Amer. Math. Soc.} 13 (2000), 149-208

\bibitem{hu}  P.Hu: Duality for smooth families in equivariant stable homotopy theory, {\em Ast\'{e}risque} 
No. 285 (2003), v+108 pp.

\bibitem{hks} P. Hu, I.Kriz, P. Somberg: Derived representation theory and stable homotopy categorification 
of $sl_k$, to appear

\bibitem{jardine} J.F.Jardine: Simplicial Presheaves, {\em J. Pure Appl. Alg.} 47 (1987) 34-87

\bibitem{kan} D. Kan: Semisimplicial spectra,  {\em Illinois J. Math.} 7 (1963) 463-478

\bibitem{lms} L.G.Lewis, J.P.May, M.Steinberger:
{\em Equivariant stable homotopy theory.
With contributions by J. E. McClure,} Lecture Notes in Mathematics, 1213. Springer-Verlag, Berlin, 1986

\bibitem{lillig} J.Lillig:
A union theorem for cofibrations,
{\em Arch. Math.} (Basel) 24 (1973), 410–415

\bibitem{mandell} M. Mandell, J.P. May, S. Schwede, B. Shipley: Model categories of diagram spectra, 
{\em Proceedings of the London Mathematical Society}, 82 (2001), 441-512

\bibitem{ms} J.P. May, J. Sigurdsson:
{\em Parametrized homotopy theory},
Mathematical Surveys and Monographs, 132. American Mathematical Society, Providence, RI, 2006

\bibitem{piacenza} R.Piacenza: Transfer in generalized prestack cohomology, {\em Pacific J. Math.} 116 (1)
1985, 185-193

\bibitem{quillen} D.G.Quillen:
{\em Homotopical algebra,}
Lecture Notes in Mathematics, No. 43 Springer-Verlag, Berlin-New York 1967

\bibitem{rg} B. Rodr\'{\i}guez Gonz\'{a}lez, A.Roig:
Godement resolutions and sheaf homotopy theory,
{\em Collect. Math.} 66 (2015), no. 3, 423-452

\bibitem{sites} Sites and sheaves, {\em Stacks project}, http://stacks.math.columbia.edu/browse

\bibitem{thomason}  R.W.Thomason: Algebraic K-theory and \'{e}tale cohomology, {\em Ann. Sci. \'{E}cole Norm. Sup.} 
(4) 18 (1985), no. 3, 437-552, Erratum:  {\em Ann. Sci. \'{E}cole Norm. Sup.} (4) 22 (1989), no. 4, 675-677

\bibitem{w1} J. H. C. Whitehead: Combinatorial homotopy. I., {\em Bull. Amer. Math. Soc.}, 55 (1949), 213-245

\bibitem{w2} J. H. C. Whitehead: Combinatorial homotopy. II., {\em Bull. Amer. Math. Soc.}, 55 (1949), 453-496

\end{thebibliography}
\end{document}